\documentclass{article}
\usepackage[utf8]{inputenc}

\title{Entropic regularisation of non-gradient systems}
\author{
\normalsize Daniel Adams\textit{$^{a}$}
       \and 
\normalsize Manh Hong Duong\textit{$^{b}$}
\and 
\normalsize Gon\c calo dos Reis\textit{$^{c,d}$}}
\date{
    \footnotesize 
    $^{a}$~Maxwell Institute for Mathematical Sciences, School of Mathematics, University of Edinburgh, Edinburgh EH9 3FD, UK. Email: d.t.s.adams@sms.ed.ac.uk
    \\
    $^{b}$~School of Mathematics, University of Birmingham, Birmingham B15 2TT, UK. Email: h.duong@bham.ac.uk
    \\
    $^{c}$~School of Mathematics, University of Edinburgh, The King's Buildings, Edinburgh, UK.
    \\
    $^{d}$~Centro de Matem\'atica e Aplica\c c$\tilde{\text{o}}$es (CMA), FCT, UNL, Portugal. Email: G.dosReis@ed.ac.uk
    \\[2ex]
        \longdate \today \ (\currenttime)
    \vspace{-0.8cm}
}

\usepackage[utf8]{inputenc}
\usepackage[T1]{fontenc}
\usepackage[english]{babel}
\usepackage{charter}

\usepackage{verbatim}
\usepackage{mathtools}
\usepackage{mathrsfs}
\usepackage{graphicx}
\usepackage{bbm}
\usepackage{framed}
\usepackage[normalem]{ulem}
\usepackage{amsthm}
\usepackage{amsmath}
\usepackage{amssymb}
\usepackage{amsfonts}
\usepackage{dsfont}
\usepackage{enumerate}
\usepackage{xcolor}
\usepackage{stmaryrd} 
\usepackage[top=1 in,bottom=1in, left=1 in, right=1 in]{geometry}
\usepackage{csquotes}
\usepackage[toc,page]{appendix}
\usepackage{colonequals}
\def\div{\mathop{\mathrm{div}}\nolimits}
\usepackage{hyperref}

\usepackage{lineno} 

\theoremstyle{definition}
\newtheorem{theorem}{Theorem}[section]
\newtheorem{lemma}[theorem]{Lemma}

\newtheorem{corollary}[theorem]{Corollary}
\newtheorem{defn}[theorem]{Definition}
\newtheorem{prop}[theorem]{Proposition}
\newtheorem{assumption}[theorem]{Assumption}

\theoremstyle{remark}
\newtheorem{rem}[theorem]{Remark}

\numberwithin{equation}{section}
\numberwithin{figure}{section}
 \usepackage[nodayofweek]{datetime}

\newcommand{\gbox}[1]{\fbox{$\triangleright$\textcolor{blue}{\textbf{Gon}:} #1}}

\usepackage{pifont}
\usepackage{
            ulem		
           ,soul		
} 

\newcommand{\divv}{\text{div}}
\newcommand{\dett}{\text{det}}
\newcommand{\argmin}{\text{argmin}}
\newcommand{\BV}{\text{BV}}

\newcommand{\bR}{\mathbb{R}}

\newcommand{\bN}{\mathbb{N}}

\newcommand{\x}{\mathbf{x}}
\newcommand{\y}{\mathbf{y}}

\newcommand{\cF}{\mathcal{F}}
\newcommand{\cG}{\mathcal{G}}

\newcommand{\cL}{\mathcal{L}}
\newcommand{\cM}{\mathcal{M}}

\newcommand{\cP}{\mathcal{P}}

\newcommand{\cT}{\mathcal{T}}

\newcommand{\cW}{\mathcal{W}}

\newcommand{\sL}{\mathscr{L}}

\newcommand{\rx}{\mathbf{x}}
\newcommand{\ry}{\mathbf{y}}
\newcommand{\rz}{\mathbf{z}}

\begin{document}
\selectlanguage{english}

\maketitle

\begin{abstract}
The theory of Wasserstein gradient flows in the space of probability measures has made an enormous progress over the last twenty years. It constitutes a unified and powerful framework in the study of dissipative partial differential equations (PDEs) providing the means to prove well-posedness, regularity, stability and quantitative convergence to the equilibrium. The recently developed entropic regularisation technique paves the way for fast and efficient numerical methods for solving these gradient flows. However, many PDEs of interest do not have a gradient flow structure and, a priori, the theory is not applicable. In this paper, we develop a time-discrete entropy regularised variational scheme for a general class of such non-gradient PDEs. We prove the convergence of the scheme and illustrate the breadth of the proposed framework with concrete examples including the non-linear kinetic Fokker-Planck (Kramers) equation and a non-linear degenerate diffusion of Kolmogorov type. \textcolor{black}{Numerical simulations are also provided.}
\end{abstract}
\vspace{0.3cm}


\scriptsize{
\tableofcontents
}
\normalsize{}
\normalsize

\section{Introduction}
In the seminal work \cite{jordan1998variational} Jordan, Otto and Kinderlehrer show that the linear Fokker-Planck Equation (FPE)
$$
\partial_t \rho = \div(\rho\nabla f)+\Delta \rho\quad \text{on}~~\mathbb{R}_+\times \mathbb{R}^d\quad\text{and}\quad \rho(0,\cdot)=\rho_0, 
$$
where the potential $f\colon\mathbb{R}^d\rightarrow [0,\infty)$ is a smooth function, can be interpreted as a gradient flow of the free energy functional with respect to the Wasserstein metric. More specifically, they prove that the solution of the FPE can be iteratively approximated by the following minimising movement (steepest descent) scheme: given a time-step $h>0$ and defining $\rho^0_h:=\rho_0$, then the solution $\rho^n_h$ at the $n$-th step, $n=1,..., \lfloor \frac{T}{h}\rfloor$, is determined as the unique minimiser of the following minimisation problem 
\begin{equation}
\label{eq discrete general grad flow}
    \min_{\rho} \frac{1}{2h} W_2^2(\rho^{n-1}_h, \rho)+\mathcal{F}_{\textrm{fpe}}(\rho),
\end{equation}
over the space of the probability measures with finite second moments. In \eqref{eq discrete general grad flow}, the free energy functional $\mathcal{F}_{\textrm{fpe}}$ is the sum of the (negative) Boltzmann entropy functional and the external energy functional, and $W_2(\cdot,\cdot)$ denotes the Wasserstein distance between two probability measures on $\mathbb{R}^d$ with finite second moments, see Section \ref{sec: notation} for detailed definition. The variational scheme \eqref{eq discrete general grad flow} is now commonly known in the literature as the `JKO-scheme'. Over the last twenty years, many PDEs have been shown to fit this Wasserstein gradient flow perspective. These include the porous medium equation \cite{otto2001geometry}, the (non-linear-non-local) Vlasov-Fokker-Planck equation \textcolor{black}{(aggregation-diffusion equation)} \cite{Carrillo2006,carrillo2003kinetic}, the fourth order quantum drift-diffusion equation and related models \cite{Gianazza2009,Matthes2009}, just to name a few. The theory of Wasserstein gradient flows creates links between different areas of mathematics such as analysis, optimal transport, and probability theory, and constitutes a unified and powerful framework in the study of dissipative PDEs providing the means to prove well-posedness, regularity, stability and quantitative convergence to the equilibrium, see the monographs \cite{ambrosio2008gradient, villani2008optimal} for great expositions of the topic.  In the last decade, the theory has been extended to a variety of different settings including general metric spaces \cite{ambrosio2008gradient}, Riemann manifolds \cite{Zhang2007}, and discrete structures \cite{Chow2012,Maas2011,Mielke2013}. More recently, it has been shown that, for many systems, the Wasserstein gradient flow structure arises from large deviation principles of the underlying stochastic processes \cite{adams2013large,adams2011large,duong2013wasserstein,Duong2013, erbar2015large}. The links between Wasserstein gradient flows and large deviation principles not only explain the origin and interpretation of such structures but also give rise to new gradient-flow structures \cite{mielke2014relation}. 

\textbf{Entropic regularisation of optimal transports and of Wasserstein gradient flows}. 
The most distinguished feature of the JKO-scheme \eqref{eq discrete general grad flow} is that it reveals explicitly the (physically relevant) free energy functional as the driving force and the Wasserstein metric as the dissipation mechanism for the Fokker-Planck equation.  There has been a growing interest in developing structure-preserving numerical methods for Wasserstein-type gradient flows using the JKO scheme \cite{burger2012regularized,carrillo2019blob, carrillo2010numerical}. However, from a computational point of view, implementing the JKO scheme \eqref{eq discrete general grad flow} directly is expensive since at each iteration it requires the resolution of a convex optimisation problem involving a Wasserstein distance to the previous step. This is a common difficulty in the computation of optimal transport problems. The entropic regularisation technique developed in \cite{cuturi2013sinkhorn} overcomes this difficulty by transforming the transport problem into a strictly convex problem that can be solved more efficiently with matrix scaling algorithms such as the Sinkhorn’s algorithm \cite{KnoppSinkhorn1967}. This regularisation technique has found applications in a variety of domains such as machine learning, image processing, graphics and biology, see the recent monograph \cite{peyre2019computational} for a great detailed account of the topic. By replacing the usual Wasserstein distance in the JKO scheme \eqref{eq discrete general grad flow} by its entropy smoothed approximation one obtains a regularised scheme for the Fokker-Planck equation. As in general entropic regularisation techniques for optimal transport problems, the regularised scheme leverages the reformulation of this smooth optimisation problem as a Kullback-Leibler projection and makes use of Dykstra's algorithm to attain a fast and convergent numerical scheme \cite{carlier2017convergence,peyre2015entropic}. Similar ideas have been applied to other evolutionary equations such as flux-limited gradient flows \cite{matthes2020discretization} and a tumour growth model of Hele-Shaw type \cite{DiMario2020}. 

\textbf{Variational formulation for non-gradient systems.} Many fundamental PDEs are not gradient flows but still posses an entropy (Lyapunov) functional. A typical example is the kinetic Fokker-Planck (Kramers) equation, which is a degenerate diffusion (the Laplacian operator acts only on the velocity variable but not on the position ones) and contains both conservative and dissipative effects \cite{kramers1940brownian,Ris84}. Due to the presence of the entropy functional, developing a variational formulation akin to the JKO-minimising movement scheme \eqref{eq discrete general grad flow} for these non-gradient systems is a natural question, but is still generally open. The main  difficulty in constructing such variational schemes is to find an appropriate (optimal transport) cost function(al), which is often non-homogeneous, time-step dependent and does not induce a metric. Nonetheless, for the kinetic Fokker-Planck equation, several schemes have been built, in which the corresponding cost functions are found based on either the fundamental solution or the conservative part \cite{duong2014conservative,Huang00}, \textcolor{black}{see also \cite{HuangJordan2000} for a similar approach for the non-linear Vlasov-Poisson-Fokker-Planck equation}.  Other interesting examples include the class of Lagrangian systems with local transport  \cite{figalli2011variational} and a class of degenerate diffusions of Kolmogorov type \cite{DuongTran18} in which the cost functions are derived respectively from the underlying Lagrangian structure and the small-noise (Freidlin-Wentzell) large deviation rate functional.\medskip

In this paper, motivated by the discussion in the previous paragraphs, we develop entropic regularisation schemes for a general class of non-gradient systems and apply the abstract framework to several concrete examples. 


\textbf{An abstract framework for non-gradient systems}. In this work we consider evolution equations of the form
\begin{equation}
\label{eq:Formal-Evolution}
\partial_t\rho =\sL^\ast \rho, \qquad 
\rho\vert_{t=0}=\rho_0,
\end{equation}
where $\sL^\ast$ is the formal (linear or non-linear) adjoint operator of the generator $\sL$ of a Markov process on a state space $\bR^d$ and the unknown $\rho$ is a time-dependent probability measure on~$\bR^d$, i.e.~$\rho:[0,T]\rightarrow \mathcal{P}(\bR^d)$. Thus Equation \eqref{eq:Formal-Evolution} can be viewed as the forward Kolmogorov equation associated to the Markov process describing the time-evolution of $\rho$. Equation \eqref{eq:Formal-Evolution} arises naturally in statistical mechanics for which $\rho(t,x)\,dx$ often models the probability of finding a particle, evolving according to the Markov process, at state $x$ and time $t$ \cite{Ris84}. We focus on systems where the operator $\sL^*$ has a general non-linear drift-diffusion form
\begin{equation}
\label{eq: general form of L*}
\sL^\ast\rho
=\div\big(b\rho\big)+\div\Big(\rho A\nabla\frac{\delta \cF}{\delta\rho}\Big),
\end{equation}
where $b\colon \bR^d\rightarrow \bR^d$ is a given vector field, $A$ is a symmetric (possibly degenerate) matrix in $\bR^{d\times d}$ and $\mathcal{F}\colon\cP(\bR^d)\rightarrow \bR$ is the free energy functional which is the sum of an internal energy and an external energy, see Section \ref{sec: notation} for a precise formulation. When $b=0$ and $A$ is non-singular, \eqref{eq:Formal-Evolution} is a (weighted) Wasserstein gradient flow \cite{lisini2009nonlinear}. However, \textcolor{black}{in general \eqref{eq:Formal-Evolution} is a non-reversible dynamics due to the fact that the drift $b$ is not necessarily a gradient (also the (constant) diffusion matrix $A$ may be degenerate) }\cite{adams2013large}.  This class covers non-gradient systems such as the non-linear kinetic Fokker-Planck equation and a non-linear degenerate diffusion equation of Kolmogorov type, which will be discussed in detail in Section \ref{section Concrete Problems} as concrete applications. 

\textbf{Entropic regularisation for non-gradient systems}. In this paper, we develop an entropic regularised variational approximation scheme for the evolution equation \eqref{eq:Formal-Evolution}.  The scheme is as follows: given a small parameter (which is the strength of the regularisation) $\varepsilon>0$ and a time-step $h>0$,  define $\rho_{h,\varepsilon}^{0}=\rho_0$ then $\rho_{h,\varepsilon}^{n}$ is iteratively (over $n=1,\ldots,N$ with $h$ such that $hN=T$) determined as the unique minimiser of the following minimisation problem
\begin{equation}
\label{eq: regularized scheme}
      \min_{\rho} \frac{1}{2h}\cW_{c_h,\varepsilon}(\rho^{n-1}_{h,\varepsilon},\rho)+\cF(\rho),
\end{equation}
over the space $\cP^r_2(\bR^d)$ of absolutely continuous probability measures with finite second moment. Here $\cW_{c_h,\varepsilon}$ is an appropriate regularised Monge-Kantorovich optimal transport cost functional 
\begin{equation}
\label{eq: regularized cost functional}
    \cW_{c_h,\varepsilon}(\mu,\nu):=\inf_{\gamma\in \Pi(\mu,\nu)}\Big\{\int_{\bR^{2d}} c_h(x,y)\gamma(dx,dy)+\varepsilon H(\gamma)\Big\},
\end{equation}
where the infimum is taken over the couplings between $\mu$ and $\nu$. In \eqref{eq: regularized cost functional}, the function $c_h:\bR^{d}\times \bR^{d}\to \bR$, which depends on the time-step $h$, should be thought of as the cost of displacing mass from point $x$ to $y$ in a time-step $h$. The regularisation term, $H(\gamma)$, is the entropy of $\gamma$. We note that no specific form for the cost $c_h$ is prescribed, instead, it is assumed to satisfy the conditions in Assumption \ref{assumption for the cost} (see below) which in turn means that $c_h$ is not necessarily a metric. To the best of our knowledge we are unaware of any general algorithm yielding $c_h$ given the generator $\sL$, nonetheless, in our examples Section \ref{section Concrete Problems} below we provide concrete methods to identify $c_h$. The minimisation problem (which is \eqref{eq: regularized scheme} for a single step), 
\begin{equation}\label{eq : minimisation problem}
    \argmin_{\nu \in \cP^r_2(\bR^d)} \Big\{ \frac{1}{2h}\cW_{c_h,\varepsilon}\big(\mu,\nu\big)+\cF(\nu) \Big\},
\end{equation}
will play an essential role in this work. The contribution of the present paper include: 
\begin{enumerate}
    \item[(i)] Proposition \ref{lemma : well-posedness of jko scheme} proves the well-posedness of the optimal transport minimisation problem \eqref{eq : minimisation problem}.
    \item[(ii)] \textit{An abstract framework.} Theorem \ref{theorem MAIN THEOREM} establishes, under certain conditions on the drift vector $b$, the diffusion matrix $A$ and the cost function $c_h$ (See Section \ref{section An abstract result} for precise statements), the convergence of the regularised scheme \eqref{eq: regularized scheme} to a weak solution of \eqref{eq:Formal-Evolution}.
    \item[(iii)] \textit{Concrete applications.} We illustrate the generality of our work in Section \ref{section Concrete Problems} by providing three examples to which our work is applicable: a non-linear diffusion equation with a general (\textcolor{black}{constant}, possibly singular) diffusion matrix, the non-linear kinetic Fokker-Planck (Kramers) equation, and a non-linear degenerate diffusion equation of Kolmogorov type. The drift vector field $b$ is not present in the first example but plays an important role in the last two cases. 
    \item[(iv)] \textcolor{black}{\textit{Numerics.} In Section \ref{sec:numerical experiment} a numerical implementation of our scheme, via a matrix scaling algorithm, is shown to solve Kramers equation.} 
\end{enumerate}
The proof of Proposition \ref{lemma : well-posedness of jko scheme} follows the standard procedures in \cite{carlier2017convergence,villani2008optimal}. We now provide \textcolor{black}{further discussion} concerning the points (ii), (iii) \textcolor{black}{and (iv)}.\
 \medskip 

\textbf{Comparison with the existing literature.} The general framework we detail in Section \ref{section An abstract result} provides a sufficient condition to guarantee the convergence of the regularised variational iterative scheme \eqref{eq: regularized scheme} to a weak solution of \eqref{eq:Formal-Evolution}.  We emphasise that the three distinguishing features of the PDE class we handle and which makes this an involved task are: the drift $b$ is not assumed to be of gradient type, $A$ can be singular and the operator $\sL^*$ can be non-linear. We have not found other works which deal with these features simultaneously (with or without regularisation). The proof of the main abstract theorem follows the now well-established procedure introduced originally in \cite{jordan1998variational}. However, due to the incorporation of the mentioned features, several technical improvements are  performed, in particular the introduction/construction of change of variable maps to deal with the non-metric essence of the cost function $c_h$ (see Assumption \ref{Assumption change of var for f and c - for regular.}). Our framework generalises several specific cases that have been studied previously in the literature.

A regularised variational scheme for the non-linear diffusion equation when the drift $b$ vanishes and the diffusion matrix $A$ is the identity matrix has been studied in \cite{carlier2017convergence}. This paper actually inspires our work and we slightly extend it to the case when $A$ is a general (possibly singular) matrix. This provides an entropic regularised scheme for weighted-Wasserstein gradient flows \cite{lisini2009nonlinear}. More importantly, as mentioned above, our framework accommodates singular diffusion coefficients. \textcolor{black}{In this vein, our work generalises, by allowing non-linear diffusions and including regularisation, previous works that develop un-regularised JKO-type variational approximation schemes for the linear kinetic Fokker-Planck (Kramers) equation \cite{duong2014conservative,Huang00} and a degenerate diffusion equation of Kolmogorov type \cite{DuongTran18}. In addition, several papers numerically investigate and implement regularised schemes for these equations but do not rigorously prove the convergence of the schemes as the regularisation strength tends to zero \cite{Caluya2021,caluya2019gradient}. Thus our present work provides a rigorous foundation for these works. We emphasise that our proof of convergence also holds true without regularisation. By introducing regularisation, our proposed schemes are also computationally tractable and useful for numerical purposes (see Section \ref{sec:numerical experiment} for discussion on the numerical implementation and illustrations).
}
\smallskip

\textbf{Outlook for future work}. The examples considered in this paper belong to a more general class of non-gradient systems, namely GENERIC (General Equation for Non-Equilibrium Reversible-Irreversible Coupling) systems \cite{Oettinger05}. The GENERIC framework has been used widely in physics and engineering, most notably to derive coarse-grained models. As indicated by its name, GENERIC systems contain both reversible dynamics and irreversible dynamics which are described via two geometric structures (a Poisson structure and a dissipative structure) and two functionals (an energy functional and an entropy functional). These operators and functionals are required to satisfy certain conditions, under which GENERIC systems automatically justify the laws of thermodynamics, namely the energy is preserved and the entropy is increasing (note that the entropy in mathematical literature is the negative of the entropy in the physics literature). The appearance of the concepts of
energy and entropy in the formulation of GENERIC suggests a strong variational connection. However, establishing a variational formulation (even unregularised) akin to the JKO-minimising movement scheme \eqref{eq discrete general grad flow}, in particular identifying a suitable cost function for GENERIC systems is still open, \textcolor{black}{although, encouraging attempts have been made recently for several systems as discussed above. Another interesting problem for future work is to develop and establish the convergence of JKO-type minimising movement schemes for (non-linear, degenerate) non-diffusive systems. For these systems, a proof following the seminal procedure in \cite{jordan1998variational}, which is employed in this paper, cannot be directly applied because the corresponding objective functional is not superlinear due to the absence of the entropy term. Thus, a delicate analysis needs to be introduced to obtain necessary compactness properties for the sequence of the discrete minimisers. Such analysis has been carried out for the transport equation \cite{KinderlehrerTudorascu2006} and its linear kinetic counterpart \cite{DuongLu2019}; however,  for more complicated systems such as the kinetic equation of granular media \cite{Agueh2016} it is still an open question.} Finally, the convergence analysis of (fully discretised) regularised schemes which possess a time-step dependent, non-homogeneous, non-metric cost function such as the ones in this paper or in \cite{figalli2011variational,peletier2020jump} has not been explored in totality. 

\bigskip

\textbf{Organisation of the paper.} In Section \ref{sec: main results} we present the framework and the main abstract result of this paper, Theorem \ref{theorem MAIN THEOREM}.   Section \ref{section Concrete Problems} outlines some explicit examples of where our work is applicable, their verification is left to the appendix. \textcolor{black}{ A numerical implementation of our scheme applied to Kramers equation is carried out and analysed in Section \ref{sec:numerical experiment}}.   Section \ref{section well posedness} contains the well-posedness of the scheme, and in Section \ref{sec:GeneralCase} we prove the main result. \textcolor{black}{In the Appendix we give proofs of some technical lemmas and verification of the examples. }

\section{Main Results}
\label{sec: main results}
In this section, we first introduce notations that will be used throughout the paper, then we present the lists of assumptions, together with their interpretations, and finally we state the main abstract result, Theorem \ref{theorem MAIN THEOREM}.  
\subsection{Notation}
\label{sec: notation}

Throughout $d\in \bN$ will be the dimension of the space. A fixed real $T>0$ denotes the length of the time interval we consider. Throughout, $C$ denotes a constant whose value may change without indication and depends on the problem's involved constants, but, critically, it is independent of key parameters of this work, namely $\varepsilon,h>0,N\in \bN$ introduced in Assumption \ref{assumption on scaling}. The Euclidean inner product will be written as $\langle \cdot , \cdot \rangle$. We write $\|\cdot\|$ as the Euclidean norm on $\bR^d$, and $|\cdot|$ when $d=1$. The symbol  $\|\cdot\|$ is also used as the $2$-norm on $\bR^{d_1\times d_2}$. For a matrix $A$ let $A^T$ be its transpose. 


The space of Lebesgue $m-$integrable functions on $\Omega \subset \bR^d$ is denoted by $L^m(\Omega)$, with norm $f\mapsto  \|f\|_{L^m(\Omega)}= \big( \int_{\Omega} \|f(x)\|^m dx \big)^{1/m}$. Let $\Omega \subset \bR^d$, the supremum norm $\|\cdot\|_{\infty,\Omega}$ of a vector field $\phi : \Omega \to \bR^d$, or a function $\phi : \Omega \to \bR$, is used to denote $\sup_{x\in \Omega}\|\phi(x)\|$, $ \sup_{x\in \Omega}|\phi(x)|$ respectively, when $\Omega=\bR^d$ we just write $\|\cdot\|_{\infty}$. 

We use an enhanced version of the Landau ``big-O'' and ``small-o'' notation in the following way: the ``big-O'' notation $\phi(h)=O(\varphi(h))$, for functions $ \phi,\varphi : \bR_+ \to \bR $ denotes that there exists $C,h_0>0$ such that $|\phi(h)|\leq C \varphi(h)$ for all $h<h_0$ and we say a matrix $B\in \bR^{d\times d}$ is $O(h)$ if $\max_{i,j}|B_{i,j}|\leq C h$ -- \color{black}critically, the constants $C,h_0$ are independent of any other parameter/variable of interest that $\phi$ or $B$ may depend on (otherwise such dependence is made explicit).

\color{black} Further we use the Landau ``little-o'' notation $\phi(h)=o(\varphi(h))$ to mean $\lim_{h\to 0} \frac{\phi(h)}{\varphi(h)}=0$.

Let $A,B\subseteq \bR^d$, define $C^k(A;B)$ as the $k-$times continuously differentiable functions from $A$ to $B$ with continuous $k^{th}$ derivative. Define $C^\infty_c(A;B)$ as the set of infinitely differentiable functions from $A$ to $B$ with compact support. 
Let $\nabla \phi$, $\Delta \phi$, and $\nabla^2 \phi$ be the gradient, Laplacian, and Hessian respectively, of a sufficiently smooth function $\phi : \bR^d \to \bR $. For a sufficiently smooth vector field $\eta : \bR^d \to \bR^d$ let $\divv (\eta)$, and $D\eta$ be its divergence and Jacobian respectively. We call the identity map $\text{id}$. 


Denote the space of Borel probability measures on $\bR^d$ as $\cP(\bR^d)$. The second moment $M$ of a measure $\rho \in \cP(\bR^d)$ is defined as
\begin{equation}\label{eq second moment definition}
    \cP(\bR^d) \ni \rho \mapsto 
M(\rho):=\int_{\bR^d}\|x\|^2\rho(dx).
\end{equation}
The set of probability measures with finite second moments is denoted by $\cP_2(\bR^d)$, 
\begin{equation}
\label{eq: measure second moment space definition}
    \cP_2(\bR^d):=\{\rho \in \cP(\bR^d)~:~M(\rho)<\infty\}.
\end{equation}
Define $\cP_2^r(\bR^d)$ as those $\rho \in \cP_2(\bR^d)$ which are absolutely continuous with respect to the Lebesgue measure. We will use the same symbol $\rho$ to denote a measure $\rho \in \cP_2^r(\bR^d)$ as well as its associated density. Define $H$ to be the negative of Boltzmann entropy,
\begin{equation}\label{eq bolztman entropy definition }
    \cP(\bR^d) \ni \rho \mapsto 
H(\rho):= \begin{cases}
    \int_{\mathbb{R}^{d}} \rho \log \rho, &\text{if $\rho \in \cP^r_2(\bR^d)$}
    \\
    +\infty, & \textrm{otherwise}
    \end{cases},
\end{equation}
 which throughout we will just refer to as the entropy. 

The set of transport plans between given measures $\mu,\nu \in \cP_2(\bR^d)$ is denoted by $\Pi(\mu,\nu)\subset \cP_2(\bR^{2d})$. That is, for $\mu,\nu \in \cP_2(\bR^d)$, $\gamma \in \Pi(\mu,\nu)$ if $\gamma(\mathcal{B}\times \bR^d)=\mu(\mathcal{B})$ and $\gamma( \bR^d \times \mathcal{B})=\nu(\mathcal{B})$ for all Borel sets $\mathcal{B}\subset \bR^d$. Let $\Pi^r(\mu,\nu)$ be those $\gamma \in \Pi(\mu,\nu)$ which are absolutely continuous. Throughout, when a measure is said to be `absolutely continuous' we implicitly mean with respect to the Lebesgue measure. We denote a sequence of probability measures indexed by $k\in \bN$ as $\{\mu_k\}_{k\in \bN}$ which we relax to $\{\mu_k\}$. We use the symbol $\rightharpoonup$ to mean the weak convergence of measures. For any two subsets $P,Q\subset \cP_2(\bR^d)$ we denote $\Pi(P,Q)$ as the set of transport plans whose marginals lie in $P$ and $Q$ respectively. For a vector field $\eta:\bR^d\to \bR^d$ and measure $\mu\in\cP(\bR^d)$ we write $(\eta)_{\#}\mu$ as the push-forward of $\mu$ by $\eta$. For any probability measure $\gamma$ and function $c$ on $\bR^{2d}$ we write
$$
(c,\gamma):=\int_{\bR^{2d}} c(x,y) \gamma(dx,dy).
$$
Lastly, the $2$-Wasserstein distance on $\cP_2(\bR^d)$ is denoted by $W_2$.

\subsection{The abstract framework and the main result}
\label{section An abstract result}
 
In this section we present the working assumptions of our abstract framework, namely, the assumptions placed on the operator $\sL^*$ \eqref{eq: general form of L*}, and transport cost $c_h$, which are assumed to hold throughout. Under these assumptions the regularised scheme \eqref{eq: regularized scheme} can be shown to be well-posed and to converge to the weak solution of the evolution equation \eqref{eq:Formal-Evolution}. 

\begin{assumption} [Free energy]
\label{Assumption on potential and internal energy} 
We assume there is a fixed constant $C>0$ such that the following holds. 
The free energy functional $\cF : \cP^r_2(\bR^d)\to \bR$ is the sum of a potential energy and an internal energy functional
\begin{equation}
    \label{eq: general free energy}
    \cF(\rho)=F(\rho)+U(\rho), 
\end{equation}
with 
\begin{equation*}
   F(\rho)= \int f(x) \rho(x)dx, \quad \text{ and } \quad  U(\rho)=\int u(\rho(x))\,dx.
\end{equation*}
The internal energy function  $u:[0,\infty)\to \bR$ is twice differentiable $u\in \textcolor{black}{C^2((0,\infty);\bR)}$, convex, $u(0)=0$, superlinear
 \begin{equation*}
     \lim_{s\to \infty} \frac{u(s)}{s} =\infty,
 \end{equation*}
and there exists $\frac{d}{d+2}<\alpha<1$ such that 
\begin{equation}\label{eq : an assumption on u}
    u(s)\geq -Cs^\alpha.
\end{equation}
Moreover, for any $s\in[0,\infty)$ we call $p(s):=u'(s)s-u(s)$ the pressure associated to $U$, and assume there exists some $m\in \bN$ such that 
 \begin{equation}\label{additional assumptions 3}
      p(s)\leq C s^m,
      \quad \text{and}\quad 
      p'(s)\geq \frac{s^{m-1}}{C},
    \end{equation}
  and 
   \begin{equation}\label{additional assumptions 2}
        \frac{1}{C}\int_{\bR^d} (\rho(x))^m dx\leq C M(\rho)+U(\rho), \quad \forall \rho \in \cP^r_2(\bR^d).
    \end{equation}
The potential energy $f\in C(\bR^d)$ is assumed to be non-negative $f(x) \geq 0$, and Lipschitz 
\begin{equation}
\label{assumption f Lipschitz}
    |f(x)-f(y)|\leq C \|x-y\|,\qquad \forall x,y\in \bR^d.
\end{equation}
\end{assumption}
Using the formula of the free energy, \eqref{eq:Formal-Evolution} can be written explicitly in terms of the drift $b$, the diffusion matrix $A$, the potential $f$ and the pressure $p$ as follows
\begin{equation*}
    \partial_t\rho=\sL^\ast\rho
=\div\big(b\rho\big)+\div\Big[A\Big(\nabla p(\rho)+\rho\nabla f\Big)\Big].
\end{equation*}
 \begin{rem}
To comment on the scope of Assumption \ref{Assumption on potential and internal energy}, note that the convexity and superlinear growth at infinity  of $u$  ensure that the functional $U$ is lower semi-continuous with respect to the weak convergence of measures, see Lemma \ref{lemma : appendix : l.s.c of internal energy}. \eqref{eq : an assumption on u} implies that the negative part of $u(\rho)$ is in $L^1(\bR^d)$ (for $\rho\in\cP_2(\bR^d)$). The infinitesimal pressure is modelled by $p$ and is clearly non-negative and increasing, we refer to \cite[Chapter 15]{villani2008optimal} for a further discussion. Its structure,  {\eqref{additional assumptions 3}}, allows for a large class of internal energy functionals $U$, capturing in particular the cases of the Boltzmann entropy and power functions.  

It 
is natural for the potential $f$ to be assumed bounded from below, this ensures the lower semi-continuity of $F$ with respect to weak convergence. Also, a Lipschitz $f$ means that $\frac{f(x)}{\|x\|+1}<C$ and hence $F$ will be finite.  
The aforementioned lower semi-continuity, as well as the linearity of $F$ and convexity of $U$ is the standard framework to obtain the well-posedness of the scheme. 
 \end{rem}
 
 \begin{assumption}\label{assumption on b and A}[On $b$ and $A$]
 The constant matrix $A\in \bR^{d\times d}$ is symmetric. The vector field $b\in C(\bR^d;\bR^d)$ is Lipschitz.
  \end{assumption}
  
\begin{rem}
Most notably, we allow for the matrix $A$ to be singular and the vector field $b$ to not necessarily have gradient form. This permits us to study a wider class of PDEs, see Section \ref{section Concrete Problems}. 
\textcolor{black}{When Equation \eqref{eq: general form of L*} is the Kolmogorov forward equation of the associated SDE, $A$ takes the form of the product of a diffusion matrix with its transpose, hence assuming its symmetry is natural.}  
\end{rem}
 
Next, we detail the relationship between $A$, $b$ and the cost $c_h$.
 \begin{assumption}[The cost $c_h$]
 \label{assumption for the cost}
There exists an $h_0>0$ such that for all $0<h<h_0$ the cost map $c_h:\bR^{2d}\to \bR$ is continuous and satisfies the following assumptions.
 
 \begin{enumerate}
\item Fix any $x\in \bR^d$, the map $y\mapsto c_h(x,y)$ is differentiable.

\item There exists a real valued $d\times d$-matrix $B_h$ of order $O(h)$ such that 
 \begin{equation}
 \label{equation assumption for the cost}
\big\langle  \nabla_y c_h(x,y), \tilde{\eta} \big\rangle  
-\big\langle  2(y-x) - 2h b(y) , \eta \big\rangle 
= O(h^{2})(1+\textcolor{black}{\|\eta\|})(\|x\|^2+\|y\|^2+1)+ O(1) c_h(x,y),
\end{equation}
\textcolor{black}{for all $\eta,x,y\in \bR^d$ }, where $\tilde{\eta}:=(A+B_h) \eta$.

\item  There exists a constant $C(h)>0$, possibly depending on $h$, such that 
\begin{equation}\label{eq 93}
  \|\nabla_y c_h(x,y)\|\leq C(h)\big( \|x\|^2+\|y\|^2 +1\big),\qquad \forall x,y\in \bR^d.
\end{equation}

\item There exists $C>0$ for all $x,y\in \bR^d$ such that 
\begin{equation}
\label{eq : cost realted to euclidean cost}
\|x-y\|^2 \leq C\big ( c_h(x,y)+h^2(\|x\|^2+\|y\|^2) \big),
\end{equation}
and, for some constant $C(h)>0$, possibly depending on $h$,
\begin{equation}\label{eq : cost upper bound}
    c_h(x,y)\leq C(h) \big(\|x\|^2+\|y\|^2\big),
\end{equation}
and
\begin{equation}\label{eq : cost lower bound}
    0\leq c_h(x,y).
\end{equation}
\end{enumerate}
 \end{assumption}
Before proceeding, a thorough review of this assumption is in order and we do so via the following sequence of remarks. 
\begin{rem}\ \label{rem:comments on assumption} 
\begin{enumerate} 
   \item It is the main step of the JKO procedure that motivates \eqref{equation assumption for the cost}. That is, \eqref{equation assumption for the cost} provides the essential link between the discrete Euler-Lagrange equations of our scheme (\eqref{eq Euler-Lagrange} below) and the weak solution of \eqref{eq:Formal-Evolution} (given by \eqref{eq: weak formulation general PDE} below). Equation \eqref{equation assumption for the cost} lets us replace the cost term by the drift $b$ in the discrete Euler-Lagrange equation. The RHS of \eqref{equation assumption for the cost} then guarantees that the error we make when doing this operation is still of the correct order, see Lemma \ref{lemma : general E.L equation}.
   
 
  
   \item \color{black} Conditions \eqref{eq : cost realted to euclidean cost} and \eqref{eq : cost upper bound} allow us to estimate the optimal transport cost functional $\cW_{c_h,\varepsilon}$, which is generally not a distance, in terms of the traditional Wasserstein distance. \color{black} Both \eqref{eq : cost upper bound} and \eqref{eq : cost lower bound} are natural conditions to guarantee that $\cW_{c_{h},\varepsilon}(\cdot,\cdot)$ is well defined on $\cP_2^r(\bR^d)\times\cP_2^r(\bR^d)$. The condition \eqref{eq : cost lower bound} also provides weak lower semi-continuity of $\gamma \mapsto (c_h,\gamma)$ which is essential, see the proof of Proposition \ref{lemma : well-posedness of jko scheme}, for the well posedness of the minimisation problem \eqref{eq : minimisation problem}. Again, the constant $C(h)$ may blow up as $h\to 0$. 
   \item Condition \eqref{eq 93} will be used to obtain a strong convergence for the (non-linear) pressure term when establishing the convergence of the scheme by passing to the limit $h\rightarrow 0$. Specifically, for each fixed $h>0$ \eqref{eq 93} guarantees integrability of $\|\nabla_yc_h\|$ against measures in $\cP_2(\bR^{2d})$. 
\end{enumerate}
\end{rem} 

We now remark on the generality of the cost map $c_h$.
\begin{rem}[The generality of the cost $c_h$ and concrete Examples]
Notably, the cost is \textit{not} restricted to those of the form $c_h(x,y)=c_h(x-y)$ with $c_h(x,x)= 0$, indeed such costs are usually associated to gradient flows \cite{agueh2005existence,jordan1998variational,lisini2009nonlinear}. It is clear that Assumption \ref{assumption for the cost} is verifiable in the case of $b=0$, $A$ symmetric non-singular, and  $c_h(x,y)=\langle {A^{-1}}(x-y), x-y \rangle$ the weighted Wasserstein. Indeed in \eqref{equation assumption for the cost} one can pick $B_h=0$, and obtain the exact equation 
$$
\Big\langle \nabla_y c_h(x,y), A \eta \Big\rangle   =\Big\langle  2(y-x) , \eta \Big\rangle.
$$
We claim that many fundamental non-linear PDEs will fit the structure of Assumption \ref{assumption for the cost}, and refer the reader to Section \ref{section Concrete Problems} for illustrative examples.  
 \end{rem}

\begin{assumption}[The regularisation]\label{Assumption change of var for f and c - for regular.}

  For each $h>0$ there exists a function $\mathcal{T}_h:\bR^d\to\bR^d$, called henceforth a `change of variable', such that for some $\beta >0$ and any $\sigma>0$, $z,x\in \bR^d$ 
\begin{equation}\label{kramer : eq : equation 61, with external}
    c_h(x,\mathcal{T}_h(x)+\sigma z) 
    \leq C\Big( \frac{\sigma}{h^\beta }\big(\|z\|^2+1\big) +h^2\big(\|x\|^2+1\big)\Big),
\end{equation}
and
\begin{equation}\label{eq 81}
    \big|f(\cT_h(x)+\sigma z)-f(x)\big|
    \leq C\Big( \frac{\sigma}{h^\beta }\big(\|z\|^2+1\big) +h\big(\|x\|^2+1\big)\Big),
\end{equation}
and the partial derivatives of  $\mathcal{T}_h$ are assumed 
continuous. 

\end{assumption} 

\begin{rem}
\color{black}
The above change of variables is used in Lemma \ref{lemma : change of variable satisfy assumption} to construct an admissible plan in the entropy regularised minimisation problem, allowing one to obtain a priori estimates which are crucial in establishing the convergence of the scheme. \color{black} Although the above assumption may seem burdensome to check, in practice it is not. In the classical case $c_h(x,y)=\|x-y\|^2$ one simply takes $\cT_h(x)=x$. Other examples of $\cT_h$ are given in Section \ref{section Concrete Problems}, where its clear that \eqref{eq 81} will be straightforward since $f$ is assumed Lipschitz.
\end{rem}

\begin{assumption}[The regularisation's scaling parameters]
\label{assumption on scaling}
Take three sequences $\{N_k\}_{k\in \bN}\subset \bN$, $\{\varepsilon_k\}_{k\in \bN}\subset \bR_+$, and $\{h_k\}_{k\in \bN} \subset \bR_+$, which, for any $k\in \bN$, abide by the following scaling

\begin{equation}\label{eq: assumption on scaling}
   h_k N_k = T,~~~\text{and}~~~ 0< \varepsilon_k \leq \varepsilon_k|\log \varepsilon_k|\leq C h_k^2,
\end{equation}
and are such that $h_k,\varepsilon_k\to 0$ and $N_k\to \infty$ as ${k\to \infty}$.
\end{assumption}

\begin{rem}
 The scaling \eqref{eq: assumption on scaling} is a 
 theoretical constraint introduced in \cite{carlier2017convergence} for the convergence of the JKO procedure. It ensures that the entropic regularisation is sufficiently small such that the error made by its introduction in the optimal transport problem is lost in the limit $k\to \infty$. 
\end{rem}

In this work, we are interested in weak solutions to \eqref{eq:Formal-Evolution} as defined next. 
\begin{defn}[Weak solutions]
\label{def: weak formulation general PDE}
A function $\rho\in L^1(\bR^+\times\bR^d)$, with \textcolor{black}{$p(\rho)\in L^1(\bR^+\times \bR^d)$}, is called a weak solution of Equation~\eqref{eq:Formal-Evolution} with initial datum $\rho_0\in \cP^r_2(\bR^d)$ if it satisfies the following weak formulation
\begin{align}
\label{eq: weak formulation general PDE}
&\int_0^T\int_{\bR^d} \partial_t\varphi \rho dx\,dt+\int_0^T\int_{\bR^d}(\sL \varphi)\rho dx\,dt=-\int_{\bR^d}\varphi(x)\rho_0 dx,\quad\text{for all}\quad \varphi\in C_c^\infty(\bR\times \bR^d),
\end{align}
concretely, using the form of $\cL$ \eqref{eq: general form of L*}, 
\begin{align*}
\int_0^T\int_{\bR^d} \partial_t\varphi \rho(dx)\,dt=&-\int_{\bR^d}\varphi(x)\rho_0(dx)+ \int_0^T \int_{\bR^{d}} \rho(t,x) \Big(  \Big\langle A \nabla f(x)  , \nabla \varphi (t,x) \Big\rangle - \Big\langle b(x) , \nabla \varphi(t,x) \Big\rangle \Big) dxdt
\\
& \textcolor{black}{-\int_0^T \int_{\bR^{d}}p(\rho(t,x)) \divv \Big(A\nabla \varphi (t,x)\Big)  dxdt,}\quad\text{for all}\quad \varphi\in C_c^\infty(\bR\times \bR^d).
\end{align*}
\end{defn}

The main (abstract) result of the paper is the following theorem which holds under all the above assumptions. 

\begin{theorem}\label{theorem MAIN THEOREM}[Convergence of the entropic regularisation scheme]
Let $\rho_0\in \cP^r_2(\bR^d)$ satisfy $\cF(\rho_0)<\infty$. Let $k\in\bN$ and take  $\{\rho^{n}_{\varepsilon_{k},h_{k}}\}_{n=0}^{N_{k}}$ to be the solution of the entropic regularisation scheme \eqref{eq: regularized scheme}. Define the piecewise constant interpolation $\rho_{\varepsilon_{k},h_{k}}: (0,\infty)\times \bR^d\rightarrow [0,\infty) $ by
\begin{equation}
    \label{eq: time interpolation}
   \rho_{\varepsilon_{k},h_{k}}(t)
   :=\rho^{n+1}_{\varepsilon_k,h_k} \quad \text{when}\quad \textcolor{black}{t\in[nh_k,(n+1)h_k)}.
\end{equation}
Suppose that Assumptions \ref{Assumption on potential and internal energy}, \ref{assumption on b and A}, \ref{assumption for the cost}, \ref{Assumption change of var for f and c - for regular.}, and \ref{assumption on scaling} hold. Then, as $k \rightarrow \infty$, we have the following convergence up to a subsequence 

\begin{equation*}
    \label{eq: convergence}
    \rho_{\varepsilon_{k},h_{k}}
    \rightarrow \rho \quad\text{in}\quad L^1((0,T)\times \bR^d),
\end{equation*}
where $\rho$ is a weak solution of the evolution equation \eqref{eq:Formal-Evolution}-\eqref{eq: general form of L*} in the sense of Definition \ref{def: weak formulation general PDE}.
\end{theorem}
The proof of this theorem is given in Section \ref{section proof of the main result}. In the next section we provide immediately several examples of interest as an illustration of our main results.

\begin{rem}
We do not prove uniqueness of the weak solution \eqref{eq: weak formulation general PDE} in the general setting, however \textcolor{black}{if uniqueness holds then Theorem \ref{theorem MAIN THEOREM} ensures that there is full convergence of the sequence}. In some cases the uniqueness has already been proved, for instance, if $A$ is the identity $b=0$ and $\cF$ is $\lambda$-displacement convex \cite{ambrosio2008gradient}, or in the case of the Kinetic FPE \cite{Huang00}. 
\end{rem}


\section{Concrete Problems}\label{section Concrete Problems}
Theorem \ref{theorem MAIN THEOREM} gives a general framework in which one can check if the evolution equation \eqref{eq:Formal-Evolution}-\eqref{eq: general form of L*} can be  approximated by the regularised JKO-type variational scheme \eqref{eq: regularized scheme}. Our setup does not immediately provide the cost or the change of variables, this has to be done on a case by case basis. In this section we present a number of examples showcasing the scope of Theorem \ref{theorem MAIN THEOREM}. In each case an explicit cost $c_h$, approximation matrix $B_h$, and change of variables $\cT_h$ are provided, these are then shown to satisfy Assumptions \ref{assumption for the cost} and \ref{Assumption change of var for f and c - for regular.}. In the following examples it is clear that the challenging part is identifying $c_h$ and $B_h$, whereas the change of variables usually comes for free.

The examples below make ample use of Theorem \ref{theorem MAIN THEOREM}, and thus the proofs of the statements for each example are by verification of the several assumptions of the main theorem. We thus, provide the example and results, and postpone the (sometimes tedious) verification to the corresponding Appendix.

\subsection{Non-linear diffusion equations: an illustrative toy example}
\label{sectoion example non linear diff}

In the case that $b=0$ \eqref{eq: general form of L*} becomes the non-linear diffusion equation 
\begin{equation}
\label{eq concrete non-linear diffusion equation}
    \partial_t \rho = \divv \Big( \rho A\big(\frac{\nabla p(\rho)}{\rho}+\nabla f\big) \Big).
\end{equation}
A prototypical example of \eqref{eq concrete non-linear diffusion equation} is the Porous Medium Equation $\partial_t \rho=\Delta \rho^m$, corresponding to $f=0$, $p(\rho)=\frac{\rho^m}{m-1}$ and $A$ is the identity matrix. Equation \eqref{eq concrete non-linear diffusion equation} models non-linear diffusion with drift in homogeneous 
anisotropic material. 
In \cite{lisini2009nonlinear} the author proved the convergence of a weighted-Wasserstein variational approximation scheme for \eqref{eq concrete non-linear diffusion equation} when $A$ is symmetric non-singular, non-constant, and elliptic. In \cite{carlier2017convergence} the authors proved the convergence of an entropic regularised scheme for \eqref{eq concrete non-linear diffusion equation} when $A$ is the identity matrix, in this respect, the following Proposition \ref{prop : the non-linear diffusion i.e continuity equation} extends their work. Therefore we only use this as an illustrative toy example of Theorem \ref{theorem MAIN THEOREM} in action. However, note that we allow the diffusion matrix $A$ to be possibly singular, this means that \eqref{eq concrete non-linear diffusion equation} can be degenerate in (at least) one direction.
Our strategy is to proceed via a viscosity approach. That is we perturb our system such that the choice of an appropriate cost is obvious, and so that in the limit the original system is retained.

\begin{prop}\label{prop : the non-linear diffusion i.e continuity equation}
Let $A$ be symmetric and positive semi-definite, let $b=0$. Define the free energy $\cF$ by \eqref{eq: general free energy} and let $f,p$ satisfy Assumption \ref{Assumption on potential and internal energy}. Let $\rho_0\in \cP^r_2(\bR^d)$ satisfy $\cF(\rho_0)<\infty$. 

Define the cost $c_h:\mathbb{R}^{2d}  \to \bR$ as  
\begin{equation}
    c_h(x,y):=\langle {(A+hI)^{-1}}(x-y),x-y\rangle.
\end{equation}
Let $k\in\bN$ and take $\{\rho^{n}_{\varepsilon_{k},h_{k}}\}_{n=0}^{N_{k}}$ to be the solution of the entropy regularised  scheme \eqref{eq: regularized scheme} with $c_h$ and $\cF$ as defined above. Define the associated piecewise constant interpolation $\rho_{\varepsilon_{k},h_{k}}: (0,\infty)\times \bR^d\rightarrow [0,\infty) $ as in \eqref{eq: time interpolation}. 

Then, as $k \rightarrow \infty$, with $N_k,h_k,\varepsilon_k$ abiding by Assumption \ref{eq: assumption on scaling}, we have 
\begin{equation}
    \rho_{\varepsilon_{k},h_{k}} \rightarrow \rho \quad\text{in}\quad L^1((0,T)\times \bR^d),
\end{equation}
where $\rho$ is a weak solution of the evolution equation \eqref{eq concrete non-linear diffusion equation}, with initial datum $\rho_0$, 
\begin{align}
\nonumber
\int_0^T\int_{\bR^d} \partial_t\varphi(t,x) \rho(t,x)dx\,dt=&-\int_{\bR^d}\varphi(0,x)\rho_0(x)dx+ \int_0^T \int_{\bR^{d}} \rho(t,x) \Big(  \Big\langle A \nabla f(x)  , \nabla \varphi (t,x) \Big\rangle dxdt
\\
&+\int_0^T \int_{\bR^{d}} \Big\langle A \nabla p(\rho(t,x)) ,\nabla \varphi (t,x) \Big\rangle  dxdt\quad\text{for all}\quad \varphi\in C_c^\infty(\bR\times \bR^d).
\end{align}
\end{prop}
The proof of the proposition is given in Appendix \ref{section appendix non-linear diffusion}.

\subsection{The non-linear kinetic Fokker-Planck (Kramers) equation}
\label{section example kramer}

 Let the dimension $d=2\tilde{d}$, and the vector field $b$ and diffusion matrix $A$ be given by
 
\begin{equation}\label{eq b and A for KFPE}
b(x,v)=\begin{pmatrix}
    -v
    \\
    \nabla_x g(x)
    \end{pmatrix},
    \qquad 
f(x,v)=f(v),
\qquad
A=
\begin{pmatrix}
0 & 0 
\\
0 & I
\end{pmatrix},  
\end{equation}
for some $g:\bR^{\tilde{d}}\to \bR$, and where, in the matrix $A$,  $I$ is the $\tilde{d}\times\tilde{d}$-dimensional identity matrix and $0$ stands for a $\tilde{d} \times \tilde{d}$-matrix of zeros. 
Substituting the above into \eqref{eq: general form of L*} one obtains the non-linear Kinetic FPE, 
\begin{equation}\label{eq non linear KFPE}
    \partial_t \rho=-\text{div}_x \Big( \rho v \Big)+\text{div}_{v}\Big(\rho\nabla_{x}g(x)\Big) +\text{div}_{v}\Big(\rho \nabla_{v}f(v)\Big) + \Delta_v p\big(\rho\big).
\end{equation}
If $p(\cdot)$ is the identity map, \eqref{eq non linear KFPE} reduces to the classical Kinetic FPE equation 

\begin{equation}\label{eq KFPE}
    \partial_t\rho =-\text{div}_x \Big( \rho  v \Big)+\text{div}_{v}\Big(\rho \nabla_{x}g(x)\Big) +\text{div}_{v}\Big(\rho \nabla_{v}f(v)\Big) + \Delta_v \rho,
\end{equation}

where $\rho$ describes the density of a Brownian particle with inertia 
\begin{align}\label{eq kinetic SDE}
    dX(t)&=V(t)dt,
    \\
    dV(t)&=-\nabla g(X(t)) dt - \nabla f(V(t))dt+\sqrt{2} dW(t).
    \nonumber
\end{align}
\textcolor{black}{This models the motion of a particle under the influence of three forces, an external force (the term $-\nabla g$), a frictional force (the term $-\nabla f$) and a stochastic noise captured by a $\tilde{d}$-dimensional Brownian Motion $W(t)$. The kinetic FPE \eqref{eq KFPE} contains both conservative and dissipative dynamics which can be easily understood from \eqref{eq kinetic SDE}. Ignoring the last two terms of \eqref{eq kinetic SDE} one has a Hamiltonian system with Hamiltonian energy $H(x,v)={\|v\|^2}/{2}+g(x)$. On the other hand, the frictional and noise terms are dissipative, modelling the collisions of the Brownian particle with the surrounding solvent. For a discussion on the applications of \eqref{eq KFPE} see \cite{risken1989fokker}, one of these applications being a simplified model of chemical reactions, which is the context in which Kramer \cite{kramers1940brownian} originally introduced it. In this paper, we will be interested in \eqref{eq non linear KFPE} for a non-linear pressure $p$, this can be derived via generalised thermodynamical theory \cite{ChavanisPRE2003}, motivated by the non-universality of the Boltzmann distribution. It has found applications in a wide variety of fields: physics, astrophysics, biology, \cite{chavanis2006nonlinear,chavanis2004chapman}. Because of the mixed dynamics the kinetic FPE is not a gradient flow. In addition, it is a degenerate diffusion due to the fact that the noise is present only in the velocity variable. Unregularised (one-step) variational approximation schemes for the linear kinetic FPE \eqref{eq KFPE} have been developed in \cite{duong2014conservative, Huang00}. A similar approach for the Vlasov–Poisson–Fokker–Planck systems was conducted in \cite{HuangJordan2000}. In addition, operator-splitting schemes, which consist of a transport (Hamiltonian flow) step and a steepest descent step, for \eqref{eq KFPE} have also been developed \cite{duong2014conservative, marcos2020solutions}, see also similar results for the non-linear non-local Fokker-Planck equation \cite{carlier2017splitting} and the Boltzmann equation \cite{carlen2004solution}.}


Since the pressure is incorporated into the free energy, using Theorem \ref{theorem MAIN THEOREM} one can develop a variational scheme for \eqref{eq non linear KFPE} using the cost functions derived in \cite{duong2014conservative}. Our extension of \cite{duong2014conservative} is twofold, firstly the scheme has been regularised, and secondly we allow for a non-linear pressure term $p$. Including regularisation and a non-linear pressure would make the calculations in \cite{duong2014conservative} more delicate, this added difficulty is incorporated via Theorem \ref{theorem MAIN THEOREM}. 

\begin{assumption}\label{assumption on hamiltonian force in KFPE}
Assume that $g\in C^{3}(\bR^{\tilde{d}})$ is \textcolor{black}{bounded from below} and there exists a constant $C>0$ for all $x_1,x_2\in \bR^{\tilde{d}}$,
\begin{align}
    \frac{1}{C}\| x_1-x_2 \|^2 \leq& \Big\langle x_1-x_2, \nabla g(x_1)-\nabla g(x_2) \Big\rangle, 
\label{eq assumption on hamiltonian force in KFPE eq 1}
    \\
    \|\nabla g(x_1)-\nabla g(x_2)\| \leq& C \|x_1-x_2\|, \label{eq assumption on hamiltonian force in KFPE eq 2}
    \\
    \| \nabla^2 g(x_1)\| , \| \nabla^3 g(x_1)\| \leq& C. \label{eq assumption on hamiltonian force in KFPE eq 3}
\end{align}
\textcolor{black}{We note that \eqref{eq assumption on hamiltonian force in KFPE eq 2}-\eqref{eq assumption on hamiltonian force in KFPE eq 3} implies that $g$ has quadratic growth at infinity. Without loss of generality we assume that $g\geq 0$ and $g(0)=0$, which} implies that for any $x\in\bR^{\tilde{d}}$
$$
\|\nabla g(x)\|\leq C\|x\|.
$$
\end{assumption}
We begin by proving the convergence of the entropy regularised scheme with the cost function  \cite[{E}q.~(13)]{duong2014conservative}. \textcolor{black}{As argued in \cite{duong2014conservative}, this cost function, which is derived from large deviation theory, naturally captures the conservative-dissipative coupling of the kinetic Fokker-Planck equation}. The proof of the following proposition is given in Appendix \ref{section appendix KFPE}.

\begin{prop}\label{proposition KFPE number one}
Let $A$, $b$ and $f$ be given by \eqref{eq b and A for KFPE}, with $g$ satisfying Assumption \ref{assumption on hamiltonian force in KFPE}. Define the free energy $\cF$ by \eqref{eq: general free energy} and let $f,p$ satisfy Assumption \ref{Assumption on potential and internal energy}. Let $\rho_0\in \cP^r_2(\bR^d)$ satisfy $\cF(\rho_0)<\infty$. 

Define the cost function $c_h : \bR^{2 d}\to \bR$  (\cite[{E}q.~(13)]{duong2014conservative}) 
\begin{align}
\nonumber
& c_h(x,v;x',v')
\\
&\quad :=
h\inf\Big\{ \int_0^h \|\ddot{\xi}(t)+\nabla g (\xi(t)) \|^2dt~:~\xi\in C^2( [0,h];\bR^d),~(\xi,\dot{\xi})(0)=(x,v),~ (\xi,\dot{\xi})(h)=(x',v')  \Big\}.
\label{eq KFPE COST 1}
\end{align}
Let $k\in\bN$ and take $\{\rho^{n}_{\varepsilon_{k},h_{k}}\}_{n=0}^{N_{k}}$ to 
be the solution of the entropy regularised scheme \eqref{eq: regularized scheme} with $c_h$ and $\cF$ defined above. Define the piecewise constant interpolation $\rho_{\varepsilon_{k},h_{k}}: (0,\infty)\times \bR^d\rightarrow [0,\infty) $ as in \eqref{eq: time interpolation}. Then, as $k \rightarrow \infty$, with $N_k,h_k,\varepsilon_k$ abiding by Assumption \ref{eq: assumption on scaling}, we have 
\begin{equation*}
    \rho_{\varepsilon_{k},h_{k}} \rightarrow \rho \quad\text{in}\quad L^1((0,T)\times \bR^d),
\end{equation*}
where $\rho$ is a weak solution of the evolution equation \eqref{eq non linear KFPE} with initial datum $\rho_0$, that is
\begin{align}
\int_0^T\int_{\bR^{d}} \partial_t\varphi \rho dx dv dt
=
& \int_0^T \int_{\bR^{d}} \Big(    \langle \nabla_x g +\nabla_v f , \nabla_v \varphi \rangle -\langle v , \nabla_x \varphi \rangle + \langle  \nabla_v p(\rho) ,\nabla_v \varphi \rangle \Big) \rho dxdvdt
\nonumber
\\
&-\int_{\bR^d}\varphi(0,x,v) \rho_0 dx dv \quad\text{for all}\quad \varphi\in C_c^\infty(\bR\times \bR^d).
\label{eq103}
\end{align}
\end{prop}
From a modelling perspective \eqref{eq KFPE COST 1} is the most natural choice for a cost, however it has no explicit expression and is therefore inconvenient for practical purposes. It has been shown that the explicit cost \cite[Eq.~(15)]{duong2014conservative}, which is an approximation of \eqref{eq KFPE COST 1}, can be implemented numerically \cite{caluya2019gradient}. We now argue that we can employ Theorem \ref{theorem MAIN THEOREM} to get the convergence of the entropic regularised scheme constructed with this cost too. The proof of the following proposition is given in Appendix \ref{section appendix KFPE}.
\begin{prop}\label{proposition KFPE number 2}
Let $A$, $b$ and $f$ be given by \eqref{eq b and A for KFPE}, with $g$ satisfying Assumption \ref{assumption on hamiltonian force in KFPE}. Define the free energy $\cF$ by \eqref{eq: general free energy} and let $f,p$ satisfy Assumption \ref{Assumption on potential and internal energy}. Let $\rho_0\in \cP^r_2(\bR^d)$ satisfy $\cF(\rho_0)<\infty$.

Define the cost function $c_h:\bR^{2d}\to \bR$ by \cite[Eq.~(15)]{duong2014conservative} that is 

\begin{equation}\label{eq cost function for kramer}
 c_h(x,v;x',v'):=\|v'-v+h \nabla g(x)\|^2+12\big\| \frac{x'-x}{h} -\frac{v'+v}{2} \big\|^2.     
\end{equation}

Let $k\in\bN$ and take $\{\rho^{n}_{\varepsilon_{k},h_{k}}\}_{n=0}^{N_{k}}$ to 
be the solution of the entropy regularised  scheme \eqref{eq: regularized scheme} with $c_h$ and $\cF$ defined above. Define the piecewise constant interpolation $\rho_{\varepsilon_{k},h_{k}}: (0,\infty)\times \bR^d\rightarrow [0,\infty) $ as in \eqref{eq: time interpolation}.

Then, as $k \rightarrow \infty$, with $N_k,h_k,\varepsilon_k$ abiding by Assumption \ref{eq: assumption on scaling}, we have 
\begin{equation*}
    \rho_{\varepsilon_{k},h_{k}} \rightarrow \rho \quad\text{in}\quad L^1((0,T)\times \bR^d),
\end{equation*}
where $\rho$ is a weak solution of the evolution equation \eqref{eq non linear KFPE} with initial datum $\rho_0$, that is \eqref{eq103} also holds true.
\end{prop}

\subsection{A degenerate diffusion equation of Kolmogorov-type}\label{section example degenerate}

Let $\tilde{d},n\in \bN$, and denote $\x=\begin{pmatrix}
    x_1,
    x_2,
    \ldots,
    x_{n-1},
    x_n
    \end{pmatrix}^T$, where $x_i\in \bR^{\tilde{d}}$. Set
 $d=\tilde{d}n$, and 
\begin{equation}\label{eq coef of degenerate diffusion}
b(\rx)=-(
    x_2
    ,
    x_3
    ,
    \ldots
   ,
    x_n
    ,
    0
    )^T,
~~~~
~~~~
A=
\begin{pmatrix}
0 & 0
\\
0 & I
\end{pmatrix},
~~~~
~~~~
f(\rx)=f(x_n),
\end{equation}
where, in the matrix $A$,  $I$ is the $\tilde{d}\times \tilde{d}$-dimensional identity matrix and $0$ stands for a $\tilde{d}(n-1) \times \tilde{d}(n-1)$-matrix of zeros. Then \eqref{eq: general form of L*} reduces to the following non-linear degenerate diffusion equation of Kolmogorov type
\begin{equation}
    \label{eq: gKramers}
\partial_t\rho(t,x_1,\ldots,x_n)=-\sum_{i=2}^{n}\div_{x_{i-1}}(x_{i}\rho)+\div_{x_n}(\nabla f(x_n)\rho)+\Delta_{x_n}p(\rho),
\end{equation}
for which, using Theorem \ref{theorem MAIN THEOREM}, a weak solution will be shown to exist as the limit of a regularised variational scheme. 

To gain insight into choosing an appropriate cost function we consider the linear case where $p(\cdot)$ is the identity. In this case \eqref{eq: gKramers} becomes
\begin{equation}
\label{eq: gKramers linear}
\partial_t\rho(t,x_1,\ldots,x_n)=-\sum_{i=2}^{n}\div_{x_{i-1}}(x_{i}\rho)+\div_{x_n}(\nabla f(x_n)\rho)+\Delta_{x_n}\rho,
\end{equation}
which is the forward Kolmogorov equation of the associated stochastic differential equations
\begin{align}
&d\xi_1=\xi_2\,dt\nonumber
\\&d\xi_2=\xi_3\,dt\nonumber
\\&\quad\vdots\label{eq: SDE1}
\\&d\xi_{n-1}=\xi_{n}\,dt\nonumber
\\&d\xi_n=-\nabla f(\xi_n)\,dt +\sqrt{2}\, dW(t),\nonumber
\end{align}
where $W(t)$ is a $\tilde{d}$-dimensional Wiener process. The above system describes a system of $n$ coupled oscillators, each of them moving vertically and being connected to their nearest neighbours, the last oscillator being forced by \textcolor{black}{a friction} and a random noise. Of course the simplest cases of $n=1,n=2$ correspond to the heat equation and Kramers equation (with no background potential) respectively.  When $n>2$ these type of equations arise as models of simplified finite Markovian approximations of generalised Langevin dynamics  \cite{ottobre2011asymptotic}, %
or harmonic oscillator chains \cite{bodineau2008large,delarue2010density}.

Recently \cite{DuongTran18} showed that the fundamental solution to \eqref{eq: gKramers linear} is determined by the following minimisation problem 
\begin{equation}
\label{eq: Ch cost}
c_h(\rx, \ry):=h \inf\limits_{\xi}\int_0^h\|{\xi}^{(n)}(s)\|^{2}\,ds,
\end{equation}
where $\mathbf{x}=(x_{1},\ldots,x_{n})\in\bR^{\tilde{d}n},\mathbf{y}=(y_1,\ldots,y_n)\in\bR^{\tilde{d}n}$ and the infimum is taken over all curves $\xi\in C^{n}([0,T];\bR^{d})$ that satisfy the boundary conditions
\begin{equation}
\label{eq: boundary conditions}
(\xi,\dot{\xi},\ldots,\xi^{(n-1)})(0)=(x_1,x_{2},\ldots,x_n)\quad\text{and}\quad (\xi,\dot{\xi},\ldots,\xi^{(n-1)})(h)=(y_1,y_{2},\ldots,y_{n}).
\end{equation}

The optimal value $c_h(\rx,\ry)$ is called the \textit{mean squared derivative cost function} and has been found to be useful in the modelling and design of various real-world systems such as motor control, biometrics, online-signatures and robotics, see \cite{DuongTran17} for further discussion. 

Theorem \cite[Theorem 1.2]{DuongTran17} states that the mean square derivative cost function $c_h(\rx,\ry)$ can be written in the explicit form, \begin{align}\label{eq mean square derivative cost function}
c_h(\rx,\ry)=\,h^{2-2n}\,[\mathbf{b}(h,\rx,\ry)]^{T}\cM\mathbf{b}(h,\rx,\ry),
\end{align}
where $\mathbf{b}: \bR^+\times \bR^{2\tilde{d}n} \to \bR^{\tilde{n}d} $ and $\cM\in \bR^{2\tilde{d}n}$ are explicitly given by \eqref{eq: b and M}. Using this explicit form of the cost function, \cite[Theorem 1.4]{DuongTran18} proved the convergence of an un-regularised variational scheme to the weak solution of \eqref{eq: gKramers linear}.

In the following proposition we use the cost \eqref{eq mean square derivative cost function} to construct a variational scheme for the highly degenerate non-linear PDE \eqref{eq: gKramers}, the proof of which is in Appendix \ref{section appendix degenerate}.  Our contributions are again twofold, firstly we allow for a non-linear $p$, and secondly our scheme is regularised. 
\begin{prop}\label{proposition Kolmogorov type}
Let $A$, $f$, and $b$ be given by \eqref{eq coef of degenerate diffusion}, with $f$ satisfying Assumption \ref{Assumption on potential and internal energy}. Define $\cF$ by \eqref{eq: general free energy}. Let $\rho_0\in \cP^r_2(\bR^d)$ satisfy $\cF(\rho_0)<\infty$. Define the cost function $c_h$ by \eqref{eq mean square derivative cost function}. 

Let $k\in\bN$ and take $\{\rho^{n}_{\varepsilon_{k},h_{k}}\}_{n=0}^{N_{k}}$ to 
be the solution of the entropy regularised  scheme \eqref{eq: regularized scheme} with $c_h$ and $\cF$ defined above. Define the piecewise constant interpolation $\rho_{\varepsilon_{k},h_{k}}: (0,\infty)\times \bR^d\rightarrow [0,\infty) $ as in \eqref{eq: time interpolation}.

Then, as $k \rightarrow \infty$, with $N_k,h_k,\varepsilon_k$ abiding by Assumption \ref{eq: assumption on scaling}, we have 
\begin{equation*}
    \rho_{\varepsilon_{k},h_{k}} \rightarrow \rho \quad\text{in}\quad L^1((0,T)\times \bR^d),
\end{equation*}
where $\rho$ is a weak solution of the evolution equation \eqref{eq: gKramers}, with initial datum $\rho_0$, 
\begin{align*}
\int_0^T\int_{\bR^{d}} \partial_t\varphi \rho d\x dt=& \int_0^T \int_{\bR^{d}} \Big( - \sum_{i=2}^n \langle  x_i , \nabla_{x_{i-1}} \varphi  \rangle + \langle \nabla_{x_{n}} f(x_{n}) , \nabla_{x_{n}} \varphi \rangle + \langle \nabla_{x_{n}} p(\rho), \nabla_{x_{n}} \varphi \rangle \Big)\rho  d\x dt
\\
&-\int_{\bR^{d}}\varphi(0,\x) \rho_0 d\x\quad\text{for all}\quad \varphi\in C_c^\infty(\bR\times \bR^{d}). 
\end{align*}
\end{prop}

\textcolor{black}{
\begin{rem}
As mentioned in the introduction, all examples considered in the present paper can be cast into the GENERIC framework which describes evolution equations containing both reversible dynamics and irreversible dynamics \cite{duong2021nonreversible,Duong2013,Kraaij2020}. Due to the splitting structure, a possible alternative approach to address GENERIC systems is to construct operator-splitting schemes. Such a scheme would consist of two steps: a Hamiltonian flow step and a gradient flow (minimising movement/steepest descent) step. For evolution equations in the Wasserstein space of probability measures we expect that one would need to combine the Hamiltonian flow theory developed in \cite{Ambrosio2008} (for the first step) and the gradient flow theory \cite{ambrosio2008gradient,jordan1998variational} (for the second step). This would be a challenging problem, but see \cite{carlen2004solution, carlier2017splitting,duong2014conservative,marcos2020solutions} and our recent preprint \cite{adams2021operator} for initial attempts in this direction.
\end{rem}
}


\section{An Illustrative Numerical Experiment}
\label{sec:numerical experiment}

    

We illustrate our findings with a numerical implementation of our algorithm applied to the Kramers equation of Section \ref{section example kramer}. The matrix scaling algorithm that we use is inspired by the work \cite{carlier2017convergence,cuturi2013sinkhorn,peyre2015entropic}, which are based on entropic regularisation. Our simulations (and their quality) are on par with other results found in the literature \cite{caluya2019gradient,carlier2017convergence}.

\subsection{Discretisation and the matrix scaling algorithm}\label{Section numerics discretisation}

We first carry out a discretisation and rewriting of our general scheme \eqref{eq: regularized scheme} into a form which lends itself amenable to a numerical implementation. For a chosen $M\in\mathbb{N}$ we consider some discrete points $\{x_i\}_{i=1}^M\subset\bR^d$, which are assumed to form a uniform grid in $\bR^d$, with each grid tile having volume $\lambda>0$.

 We consider discrete probability measures $\rho$ on $\bR^d$ fully supported on this grid, which are identified by their one-to-one correspondence with the probability simplex 
\begin{equation*}
    \Sigma^M:=\Bigg\{\rho\in\bR^M_+:\sum_{i=1}^M \rho_i=1\Bigg\}.
\end{equation*}
Note the small abuse of notation where the symbol $\rho$ denotes the discrete probability measure and its corresponding element in $\Sigma^M$. The density approximation of a discrete measure $\rho$ is then taken with respect to the discrete Lebesgue measure $\Lambda:=\lambda\sum_{i=1}^M \delta_{x_i}$, and is given by the vector $\frac{1}{\lambda}\rho$.

The discrete approximation of the regularised optimal transport problem \eqref{eq: regularized cost functional} is then defined as, for any $\mu,\nu\in\Sigma^M$, 
\begin{equation}
    \overline{\cW}_{c_{h},\varepsilon}(\mu,\nu):=\inf_{\pi\in \bR^{M\times M}_+} \Bigg\{ \sum_{i,j=1}^M (c_{h})_{i,j}\pi_{i,j}+\varepsilon \pi_{i,j}\log\Big(\frac{\pi_{i,j}}{\lambda^2}\Big)~:~\pi\mathbbm{1}=\mu,\pi^T\mathbbm{1}=\nu \Bigg\},
\end{equation}
where, of course, $(c_{h})_{i,j}=c_{h}(x_i,x_j)$ and $\mathbbm{1}=(1,\ldots,1)^T\in\bR^M$. With this in hand, our discrete approximation to the JKO scheme \eqref{eq: regularized scheme} becomes: given $\varepsilon,h>0$,  and some $\rho_{h,\varepsilon}^{0}\in \Sigma^M$, then, for $n=1,\ldots,N$ with $h$ such that $hN=T$, $\rho_{h,\varepsilon}^{n}$ determined iteratively as the unique minimiser of the following (discrete version of \eqref{eq: regularized scheme})
\begin{equation}
\label{eq: regularized scheme discrete}
      \min_{\rho\in \Sigma^M} \frac{1}{2h}\overline{\cW}_{c_h,\varepsilon}(\rho^{n-1}_{h,\varepsilon},\rho)+\overline{\cF}(\rho),
\end{equation}
where $\overline{\cF}(\rho):=\sum_{i=1}^M f(x_i)\rho_i+\lambda u\big({\rho_i}/{\lambda}\big)$, since $u$ acts on the density of $\rho$ with respect to discrete Lebesgue measure. Define the Gibbs Kernel $K\in\bR^{M\times M}$ by $K_{i,j}=\exp(-\frac{c_{h}(x_i,x_j)}{\varepsilon})$. Next, due to the entropic regularisation, we can make the well-known and celebrated observation \cite{peyre2015entropic} that \eqref{eq: regularized scheme discrete} can be reformulated by iteratively taking $\rho^n_{h,\varepsilon}=\pi^T \mathbbm{1}$, where $\pi$ minimises 
\begin{equation}\label{eq: scaling algorithm}
    \min_{\pi\in \bR^{M\times M}_+} \text{KL}(\pi||K)+\cG_n(\pi\mathbbm{1})+\frac{2h}{\varepsilon}\overline{\cF}(\pi^T \mathbbm{1}),
\end{equation}
where $\text{KL}(\pi||K):=\sum_{i,j}^M\pi_{i,j}\log\big(\frac{\pi_{i,j}}{K_{i,j}}\big)-\pi_{i,j}+K_{i,j}$ stands for the Kullback-Leibler divergence (KL divergence), and 
\begin{equation*}
    \cG_n(\rho):=\begin{cases}0&~\text{if}~\rho=\rho^{n-1}_{h,\varepsilon}
    \\
    \infty& ~\text{otherwise}.
    \end{cases}
\end{equation*}

Problems taking the form \eqref{eq: scaling algorithm} can be tackled by highly parallelizable matrix scaling algorithms \cite[Algorithm 1]{chizat2018scaling}; these are a generalisation of the Sinkhorn algorithm. Moreover, for the energy functional $\cF$ that we consider, there exist relatively simple 
formulas for the computation of the projections that appear in \cite[Algorithm 1]{chizat2018scaling}. It should be noted that \cite{chizat2018scaling} considers general measure spaces, where the product measure is taken as a reference in the KL divergence. 
Since we consider a uniform grid,  for us, the discrete KL divergence with respect to the product discrete Lebesgue measure is the appropriate approximation to the continuous KL divergence. Hence, the reference measures $d\rx,d\ry$ in \cite{chizat2018scaling} can be ignored in our case as our Gibbs kernel already has the mass factors multiplying it.

\subsection{Numerical simulation of Kramers equation}
We now provide the results of our simulations for Kramers equation using a form of \cite[Algorithm 1]{chizat2018scaling} re-cast to solve minimisation problems of the type of \eqref{eq: scaling algorithm}. Note that in comparison with \cite[Section V]{caluya2019gradient} we consider a different model, and employ a different spatial discretisation for which we use a uniform grid while they use grid-points as given by the forward simulated paths (a random space grid). We study this particular equation as we have access to its explicit solution and hence we are able to quantify the scheme's error. We point out that until our work (Proposition \ref{proposition KFPE number 2}), the scheme used in \cite[Section V]{caluya2019gradient} was not theoretically justified.

The dynamics is studied in dimension $2$ and without an external potential, i.e., we consider \eqref{eq non linear KFPE} with $p$ the identity, $g=0$, and $f(v)=\frac{v^2}{2}$. That is we solve
\begin{align}
\label{eq simulation KFPE}
    &\partial_t \rho(t,x,v)=-v\partial_x \rho(t,x,v)   +\partial_{v}\big(\rho(t,x,v)v \big) + \partial^2_v \rho(t,x,v).
\end{align}
If we consider the sharp initial condition   $\rho(0,x,v)=\delta(x-x_0)\delta(v-v_0)$ for some $x_0,v_0 \in \bR$, then, defining 
\begin{align*}
    &S_1(t)= (1-e^{-2  t}),~ S_2(t)=(1-e^{- t})^2,
    \\
    &S_3(t)= 2 t-3 + 4 e^{- t}-e^{-2  t},
    \\
    &\delta_1 (x,t)=x-\big(x_0+v_0(1-e^{- t})\big), ~ \delta_2 (v,t)=v-v_0e^{- t},
\end{align*}
the Green function of \eqref{eq simulation KFPE} is (see \cite{balakrishnan2008elements}) 
\begin{equation}
    \rho_{\text{exact}}(t,x,v)=\frac{1}{2\pi  \sqrt{S_1 S_3-S_2^2}}\exp\Big\{-\frac{S_1\delta_1^2-2S_2\delta_1\delta_2+S_3\delta_2^2}{2( t-2+4e^{- t}-(t+2)e^{-2 t})}\Big\}.
\end{equation}

To avoid the Dirac singularity at $t=0$ we offset the initial time, i.e., we equip \eqref{eq simulation KFPE} with the initial condition $\rho(0)=\rho_{\text{exact}}(t_0)$ for some $t_0>0$. We simulate the entropy regularised scheme with initial condition $\rho_{\text{exact}}(t_0)$. The simulations are run on a fixed discretised grid of $[-0.5,0.5]\times[-2.4,2.4]$, using $200\times 130$ points equidistant apart, using the discretised scheme described in Section \ref{Section numerics discretisation} across three different choices of regularisation parameter $\varepsilon=0.5,0.09,0.05$. 
 %
The approximation at time $t$ is compared to the exact solution $\rho_{\text{exact}}(t+t_0)$ via the  $L^1(\Lambda)$-norm (we compare integral of the absolute value of the difference of joint densities with respect to the discrete Lebesgue measure $\Lambda$, for $\lambda=\frac{4.8}{26000}$). 

Figures \ref{figure marginals} shows the evolution of the position and velocity marginals. 
 The well-known effect of blurring on the optimal transport problem stemming from regularisation \cite{peyre2019computational} is also clear from these figures: as the regularisation increases the mass is forced to spread out. Moreover, there is a roughness, especially in the velocity marginal, which disappears as the regularisation is increased (this smooths the kink) and/or the number of grid points are increased (this reduces numerical underflow and increases overall precision, see below). 
 The latter suggests why the kink is more apparent in the velocity marginal - it is supported on a larger domain and hence requires a finer grid spacing. 
 However, this has to be balanced against the (high) computational effort induced by performing optimal transport in higher dimensions. 
For our algorithm, we are forced to have a fine grid spacing in the position component to counterbalance the $h$ appearing in the cost function (and to capture the speed of diffusion). Matching this grid spacing also in the velocity component is computationally prohibitive (with our implementation). 

Figure \ref{fig error to exact fixed epsilon decreasing h} gives a quantitative analysis of the error between our scheme and the exact solution $\rho_{\text{exact}}$ (the joint density) as a function of time. As anticipated the error reduces as the entropic blurring is decreased, and the error increases with time.

   \begin{figure}[ht]
    \centering
    \includegraphics[scale=1]{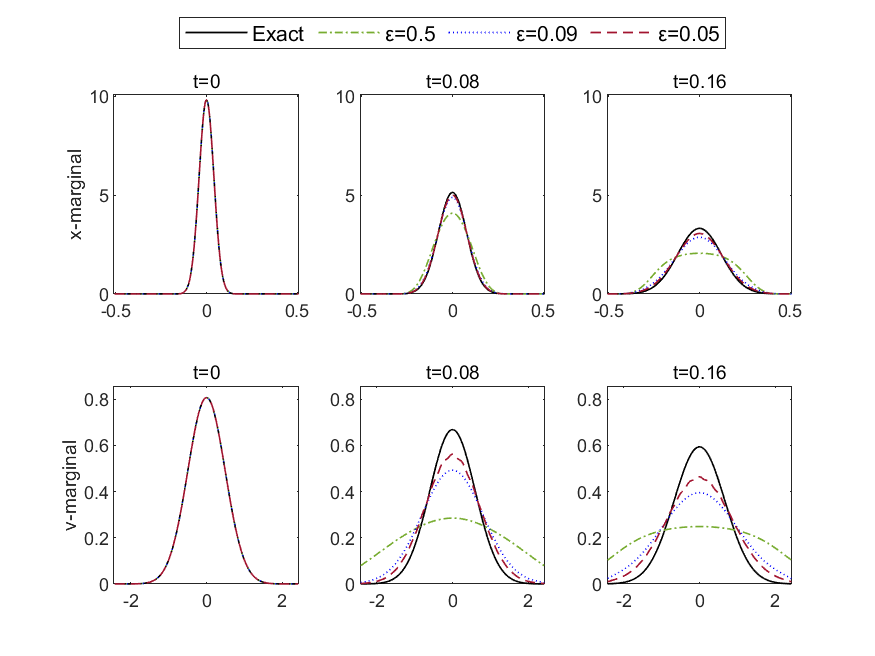}
    \caption{Comparison between the exact solution (black line) and our entropy regularised scheme for the position $x$-marginal and velocity $v$-marginal, across three time-slices $t=0,0.08,0.16$ and three regularisation choices $\varepsilon=0.5,0.09,0.05$. Simulation over the position-velocity domain $[-0.5,0.5]\times[-2.5,2.5]$. 
    All cases are ran with a step-size of $h=0.02$.
    }
    \label{figure marginals}
    \end{figure}

    \begin{figure}[th]
       \centering
    \includegraphics[scale=0.7]{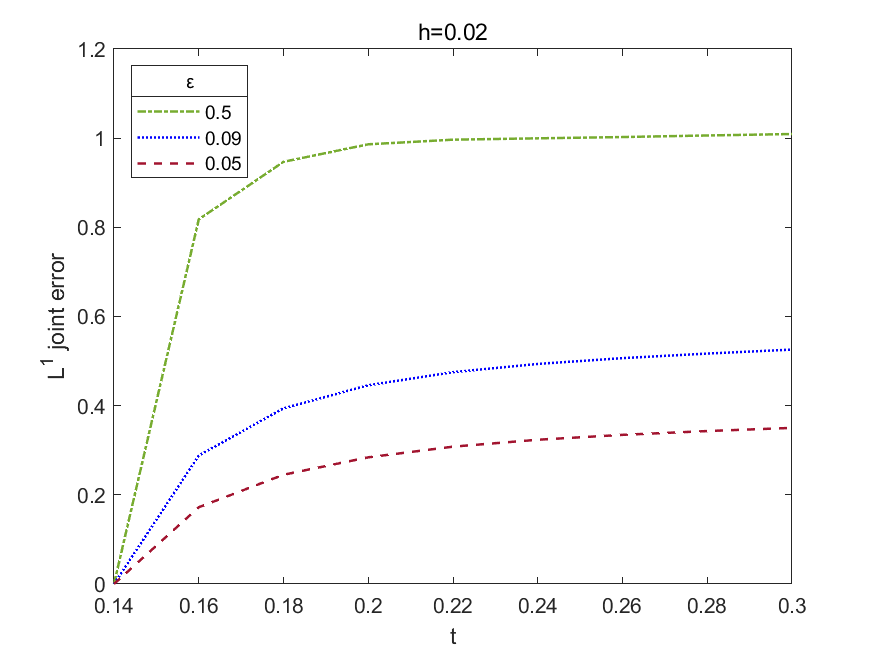}
    \caption{$L^1(\Lambda)$-norm joint error of the regularised scheme as a map of time over $[0.14,0.3]$ for multiple regularisation parameters $\varepsilon=0.5,0.09,0.05$. Simulation over the position-velocity domain $[-0.5,0.5]\times[-2.5,2.5]$. 
    All cases are ran with a step-size of $h=0.02$.}
     \label{fig error to exact fixed epsilon decreasing h}
    \end{figure}

We now discuss some of the drawbacks of the numerical implementation of this JKO type scheme. As pointed out already, regularisation introduces blurring into the system giving less sharp results. To circumvent this, one takes a small value for the regularisation parameter, however this causes numerical underflow due to the exponential form of the Gibbs Kernel $K$ (defined just above \eqref{eq: scaling algorithm}). For the vanilla Sinkhorn algorithm this is discussed in \cite[Remark 4.7]{peyre2019computational}, and for more general scaling algorithms see \cite{chizat2018scaling,schmitzer2019stabilized}. 
This issue can be partly minimised by carrying out the computations in the log-domain \cite[Section 4.4]{peyre2019computational}. Critically, the log-domain strategy is very costly due to many additional operations introduced, the algorithm is no longer just a matrix scaling algorithm. 
This issue is mitigated to a certain extent by the absorbing algorithm \cite[Algorithm 2.]{schmitzer2019stabilized}. Domain decomposition techniques \cite{bonafini2021domain} also seem a  feasible strategy to improve these algorithms.

There is a further added difficulty for schemes with a time-step dependent cost function, such as the ones introduced in our manuscript. 
Namely, for a fixed spatial discretization, as the time-step tends to zero the cost function ``blows up'', which stems from a $O(1/h^2)$-order term appearing in the cost function \eqref{eq cost function for kramer}. This (in addition to the $1/\varepsilon$ appearing in the Gibbs Kernel $K$ and discussed above) requires careful tuning, otherwise it will lead to numerical underflow. This suggests an operator-splitting scheme as in \cite{adams2021operator}, which consists of a transport (Hamiltonian flow) step and a steepest descent step capturing the conservative-dissipative structure, may be more favourable in simulating Kramers equation, since the cost function appearing in \cite{adams2021operator} is only of order $O(1/h)$ (instead of the order $1/h^2$ appearing in our cost term). 
    
Lastly, we note that in full rigour one should show the convergence of the fully discretised scheme \eqref{eq: regularized scheme discrete} to its continuous version as the volume $\lambda$ of each grid tile tends to zero. Such analysis has been done for many Wasserstein-type gradient flows \cite{BailoCarrilloMurakawaSchmidtchen2020, JungeMatthesOsberger2017,MatthesOsberger2014, matthes2020discretization}, however it is still an open question for the systems we consider here. As in the mentioned papers, we expect that some conditions, such as Courant–Friedrichs–Lewy (CFL) type condition, need to be imposed on the temporal and spatial meshes to guarantee the convergence of the fully discretised schemes. Revealing such conditions for non-gradient systems is nontrivial and we leave this question for future work.



\section{Well Posedness of the Regularised JKO scheme}\label{section well posedness}

The main result of this section is Proposition \ref{lemma : well-posedness of jko scheme}, stating the existence of a unique minimiser to the optimisation problem \eqref{eq : minimisation problem}. It is natural to achieve well-posedness of the scheme through finiteness, lower semi-continuity, and convexity of the functionals which appear in it.  There exist $h_0,\varepsilon_0>0$ depending only on the constants in our Assumptions, such that all the following results hold for $h_0>h>0,\varepsilon_0>\varepsilon>0$. Note that we are ultimately interested in the case where $h,\varepsilon \to 0$. We now give the main result of this section, the well-posedness of the optimal transport optimisation problem \eqref{eq : minimisation problem}. 

\begin{prop}\label{lemma : well-posedness of jko scheme}
Take $h,\varepsilon >0$ small enough with $\frac{\varepsilon}{h}\leq 1$ and $\mu \in \cP^r_2(\bR^d)$ with $\cF(\mu)<\infty$. Then, there exists a unique $\nu^* \in \cP^r_2(\bR^d)$ such that 
\begin{equation*}
   \nu^*  = \underset{\nu \in \cP^r_2(\bR^d)}{\argmin} \Big\{ \frac{1}{2h}\cW_{c_{h},\varepsilon}(\mu,\nu)+ \cF(\nu) \Big\}.
\end{equation*}
\end{prop}
The proof is provided at the end of the section after stating and proving a sequence of auxiliary results.

\subsection{Proofs and auxiliary results}
From \eqref{eq : cost realted to euclidean cost} in Assumption \ref{assumption for the cost} we immediately have the following result.
\begin{lemma}\label{lemma :  our cost and the wasserstein}
For any $h>0$ small enough, and any $\mu$ and $\nu$ in $\cP_2(\bR^d)$ with $\gamma$ the associated optimal plan in \eqref{eq: regularized cost functional}, it holds that 
\begin{equation*}
    M(\nu)\leq C \Big( (c_h,\gamma) + M(\mu)\Big),
\end{equation*}
where the constant $C>0$ is independent of $h,\varepsilon$.

\begin{proof}  Let $\gamma$ be optimal plan in \eqref{eq: regularized cost functional} with first marginal $\mu$ and second marginal $\nu$. Since for all $x,y \in \bR^d$ $\|y\|^2 \leq 2 (\|x\|^2+\|x-y\|^2)$, we have
\begin{align}
    M(\nu)=\int_{\bR^{2d}} \|y\|^2 d\gamma(x,y)\leq& 2 \int_{\bR^{2d}}\|x\|^2+\|x-y\|^2 d\gamma(x,y)
    \nonumber
    \\
    \leq& 2 \int_{\bR^{2d}}\|x\|^2+C \big( c_h(x,y)+h^2(\|x\|^2+\|y\|^2) \big)  d\gamma(x,y),
    \label{eq x1}
\end{align}
where in \eqref{eq x1} we have used \eqref{eq : cost realted to euclidean cost}. Hence for some $C>0$
\begin{equation*}
   M(\nu)   \leq C \Big( (c_h,\gamma)+(1+h^2)M(\mu)+h^2  M(\nu)\Big),
\end{equation*}
which implies that for small enough $h$,
\begin{equation*}
    M(\nu)\leq C \Big( (c_h,\gamma)+M(\mu)\Big).
\end{equation*}
\end{proof}
\end{lemma}
Of course if $\rho^n_{h,\varepsilon},\rho^{n-1}_{h,\varepsilon}$ are built from the scheme \eqref{eq: regularized scheme} with associated plan $\gamma^n_{h,\varepsilon}$, then Lemma \ref{lemma :  our cost and the wasserstein} says that for small enough $h$
\begin{equation}\label{eq : moments and cost function}
    M(\rho^n_{h,\varepsilon})\leq C \Big((c_h,\gamma^n_{h,\varepsilon}) + M(\rho^{n-1}_{h,\varepsilon})\Big).
\end{equation}

\begin{lemma}[Weak lower semi-continuity of $\gamma\mapsto (c_h,\gamma)$] \label{lemme l.s.c of gamma to cost}  Let $h>0$. Let $\{\gamma_k\}_{k\in \bN}\subset \cP(\bR^{2d})$, $\gamma \in \cP(\bR^{2d})$, with $\gamma_k \rightharpoonup \gamma$. Then 
\begin{equation*}
    (c_h,\gamma)\leq \liminf_{k\to \infty} (c_h,\gamma_k).
\end{equation*}
\end{lemma}
\begin{proof}
The map $c_h:\bR^{2d}\to \bR$ is continuous and non-negative by Assumption \ref{assumption for the cost}, hence the result is given by \cite[Lemma 4.3]{villani2008optimal}.
\end{proof}

\begin{lemma}[Weak lower semi-continuity of entropy under bounded 2nd moments]\label{lemma : l.s.c of H}
Let  $\{\gamma_k\}_{k\in \bN}\subset \cP_2(\bR^{2d}) $, $ \gamma \in \cP_2(\bR^{2d})$ with $\gamma_k \rightharpoonup \gamma$. Further assume that there exists a $C>0$, such that for all $k\in \bN,$  $M(\gamma_k),M(\gamma) < C$, then 
\begin{equation*}
    H(\gamma)\leq \liminf_{k\to \infty} H(\gamma_k).
\end{equation*}
\end{lemma}
\begin{proof}
This follows immediately by Lemma \ref{lemma : appendix : l.s.c of internal energy} taking $u(a)=a\log(a)$.
\end{proof}

\begin{lemma}[Existence of minimising couplings in the optimal transport problem]\label{lemma existence of minimising couplings in the optimal transport problem}
Given $\mu,\nu \in \cP^r_2(\bR^d)$ with finite entropy $H(\mu),H(\nu)<\infty$. Then, there exists a $\gamma \in \Pi^r(\mu,\nu)$ which attains the infimum in $\cW_{c_{h},\varepsilon}(\mu,\nu)$.

\begin{proof}
By \cite[Lemma 4.4]{villani2008optimal} $\Pi(\mu,\nu)$ is tight, and hence by Prokhorov's Theorem it is also relatively compact.  
Let $\gamma_{k}\in \Pi(\mu,\nu)$, $k\in \bN$, be a minimising sequence of $\cW_{c_{h},\varepsilon}(\mu,\nu)$. 

Now, using that $\Pi(\mu,\nu)$ is relatively compact, we can say (extracting a sub-sequence and relabelling) that $\gamma_k \rightharpoonup \gamma \in \Pi(\mu,\nu)$ (since $\Pi(\mu,\nu)$ is weakly closed). Lemmas \ref{lemme l.s.c of gamma to cost}, \ref{lemma : l.s.c of H} proved lower semi-continuity of $\hat{\gamma} \mapsto (c_h,\hat{\gamma})$, $\hat{\gamma}\mapsto H(\hat{\gamma})$ respectively, which implies the limit, $\gamma$, is a minimiser.

It remains only to show that $\gamma$ has a density. Using \eqref{eq : cost upper bound} (and that there exists an admissible plan, e.g.,~the product measure $\mu \otimes \nu $) we see that $\cW_{c_{h},\varepsilon}(\mu,\nu)<\infty$.  Since $\cW_{c_{h},\varepsilon}(\mu,\nu)<\infty$ and $(c_h,\gamma)\geq 0$ we deduce that $H(\gamma)<\infty$, hence $\gamma\in \Pi^r(\mu,\nu)$.
\end{proof}
\end{lemma}

So far we have shown that there exists an absolutely continuous transport plan with finite entropy that solves the optimal transport problem \eqref{eq: regularized cost functional} between any two measures in $\cP_2^r(\bR^d)$. Next, we explore some properties of the Kantorovich optimal transport cost functional $\cW_{c_{h},\varepsilon}$ defined by \eqref{eq: regularized cost functional}. 
\begin{lemma}[Strict Convexity of $\nu\mapsto \cW_{c_{h},\varepsilon}(\mu,\nu)$]
\label{lemma convexity of the transport problem}
For a fixed $\mu\in\cP_2^r(\bR^d)$,
$$ 
\cP_2^r(\bR^d) \ni \nu \mapsto \cW_{c_{h},\varepsilon}(\mu,\nu), 
$$
is strictly convex.
\begin{proof}
This follows as in \cite[Lemma 2.5]{carlier2017convergence} by linearity of $\gamma \mapsto (c_h,\gamma)$ and strict convexity of $H$.
\end{proof}
\end{lemma}

\begin{lemma}[Lower semi-continuity of $\nu \mapsto \cW_{c_{h},\varepsilon}(\mu,\nu)$ restricted to $\cP^r_2(\bR^d)$ and uniform moment bounds]
\label{lemma joint l.s.c of transport problem}
Let $\{\nu_k\}_{k\in\bN}\subset \cP_2^r(\bR^d)$, $\mu,\nu \in \cP_2^r(\bR^d)$, with $\nu_k \rightharpoonup \nu$. Moreover,  assume for all $k\in \bN$ the probability measures $\nu_k,\mu,\nu $ have uniformly bounded entropy and second moments. Then
$$ 
\cW_{c_{h},\varepsilon}(\mu,\nu) \leq \liminf_{k\to \infty} \cW_{c_{h},\varepsilon}(\mu,\nu_k). 
$$
\end{lemma}
\begin{proof}
Let $\{\nu_k\},\mu,\nu$ be as assumed above, and $\{\gamma_k\}$ be the associated optimal plans in $\cW_{c_{h},\varepsilon}(\mu,\nu_k)$. 
Note $\{\gamma_k\}\subset \Pi(\mu,\{\nu_k\})$ (see Notation section). 
Since $\{\nu_k\}$ is weakly convergent it is tight, and \cite[Lemma 4.4]{villani2008optimal} implies that $\Pi(\mu,\{\nu_k\}) $ is too, hence extracting (and relabelling) a sub-sequence $\{\gamma_k\}$, we know that $\gamma_k\rightharpoonup \gamma \in \cP(\bR^{2d})$. 
In fact $\gamma \in \Pi(\mu,\nu)$ since weak convergence of $\gamma_k$ implies weak convergence of its marginals (and we know $\nu_k \rightharpoonup \nu$). 
Now, the lower semi-continuity established in Lemmas \ref{lemme l.s.c of gamma to cost} and \ref{lemma : l.s.c of H}  implies that
\begin{align*}
  \liminf_{k\to \infty} \cW_{c_{h},\varepsilon}(\mu,\nu_k)=
  \liminf_{k\to \infty}
  \frac{1}{2h}(c_h,\gamma_k)+\varepsilon H(\gamma_k)
   &
  \geq \frac{1}{2h}(c_h,\gamma)+\varepsilon H(\gamma) 
  \\
  &\geq \cW_{c_{h},\varepsilon}(\mu,\nu).
 \end{align*} 
\end{proof}

\begin{lemma}\label{lemma l.s.c of energy}[Lower-semi continuity of $\cF$ under uniformly bounded moments]  Let $\{\mu_k\}_{k\in \bN} \subset \cP_2(\bR^d)$, $\mu \in \cP_2(\bR^d)$ with $\mu_k\rightharpoonup \mu$. Assume $\sup_k M(\mu_k) < \infty$, then 
\begin{equation}
   \cF(\mu) \leq \liminf_{k\to\infty} \cF(\mu_k).
\end{equation}
\end{lemma}
\begin{proof}
The lower semi-continuity of $U$ follows from the uniform bounded moments, Assumption \ref{Assumption on potential and internal energy} and Lemma \ref{lemma : appendix : l.s.c of internal energy}. The lower semi-continuity of $F$ follows from \cite[Theorem 2.38]{ambrosio2000functions}, since $(x,y): \bR^d \times \bR \to \bR$, $(x,y)\mapsto f(x)y$ is clearly 1-homogeneous and convex in $y$ for fixed $x$ (as $f$ is non-negative).
\end{proof}

We are now in a position to prove the main result of this section.
\begin{proof}[Proof of proposition \ref{lemma : well-posedness of jko scheme}]
Denote $J_{c_{h},\varepsilon}(\mu,\nu):= \frac{1}{2h}\cW_{c_{h},\varepsilon}(\mu,\nu)+ \cF(\nu)$, and $\gamma$ the optimal coupling in $\cW_{c_{h},\varepsilon}(\mu,\nu)$. Note that since $f\geq 0$ and by Lemma \ref{lemma : appendix : entropy bound} we have, for some fixed $C>0$ and $0<\alpha<1$,
\begin{equation}\label{eq 85}
    J_{c_{h},\varepsilon}(\mu,\nu)
    \geq 
    \frac{1}{2h}\cW_{c_{h},\varepsilon} (\mu,\nu) -C(1+M(\nu))^{\alpha}.
\end{equation}
Furthermore, since the sum of infima is less than the infima of the sum, and by the property of the entropy and marginals $H(\gamma)\geq H(\mu)+H(\nu)$, we have 
\begin{align*}
   \frac{1}{2h} \cW_{c_{h},\varepsilon} (\mu,\nu)  
    \geq 
    & 
    \frac{1}{2h}(c_h,\gamma) +\frac{\varepsilon}{2h}\big( H(\mu)+H(\nu) \big). 
\end{align*}
Moreover, using Lemma \ref{lemma :  our cost and the wasserstein} we have, for $h,\varepsilon>0$ small enough
\begin{align*}
   \frac{1}{2h} \cW_{c_{h},\varepsilon} (\mu,\nu)  \geq 
    & 
    \frac{1}{2h}(c_h,\gamma)+M(\mu)-M(\mu)+\frac{\varepsilon}{2h} \big(H(\mu) +  H(\nu)\big)
    \\
    \geq
    & C_1 M(\nu) + C_{\mu,\varepsilon,h} + \frac{\varepsilon}{2h} H(\nu), \qquad 
\end{align*}
with fixed constants $C_1>0$, and $C_{\mu,\varepsilon,h}$ depending only on $\mu,\varepsilon,h$. Consequently by Lemma \ref{lemma : appendix : entropy bound} we arrive at 
\begin{align}\label{eq 84}
    \frac{1}{2h} \cW_{c_{h},\varepsilon} (\mu,\nu)  \geq &C_1 M(\nu) +\textcolor{black}{C_{\mu,\varepsilon,h}}  - \frac{\varepsilon}{2h} C(1+M(\nu))^\alpha.
\end{align}
Combining \eqref{eq 84} with \eqref{eq 85}, and \textcolor{black}{choosing $h,\varepsilon$ small enough} 
we get that
\begin{align}\label{eq 86}
     J_{c_{h},\varepsilon}(\mu,\nu)\geq& C_1 M(\nu)+\textcolor{black}{C_{\mu,\varepsilon,h}}-C_1(1+M(\nu))^\alpha.
\end{align}
Now employing \eqref{eq appendix 1}, as well as the Bernoulli inequality: $(1+s)^\alpha \leq 1+ \alpha s$ for all $s\geq -1$ and $\alpha \in(0,1)$, one can see that \eqref{eq 86} implies that the functional $\nu\mapsto J_{c_{h},\varepsilon}(\mu,\nu)$ is bounded from below. Note that there exists a $\nu\in\cP_2^r(\bR^d)$ such that  $J_{c_{h},\varepsilon}(\mu,\nu)<\infty$, for example,  take $\nu=\mu$ (and the product plan). Let $\{\nu_k\}$ be a minimising sequence of $\nu\mapsto J_{c_{h},\varepsilon}(\mu,\nu)$. Note $M(\nu_k),H(\nu_k)$ are uniformly bounded. Since $M(\nu_k)$ is uniformly bounded,   
%
the set $\{\nu_k\}$ is tight
, hence extracting a subsequence (not relabelled) we obtain $\nu_k\rightharpoonup \nu \in \cP(\bR^d)$. Moreover, $\nu \in \cP_2(\bR^d)$ since uniform bounded 2nd moments and weak convergence implies the limit has a bounded 2nd moment. 
The lower semi-continuity proved in Lemmas \ref{lemma joint l.s.c of transport problem} and \ref{lemma l.s.c of energy} ensures that the limit $\nu$ is a minimiser. That $\nu \in \cP^r_2(\bR^d)$ follows since lower semi-continuity of $\cP_2(\bR^d) \ni \nu \mapsto H(\nu)$ (see Lemma \ref{lemma : appendix : l.s.c of internal energy}) which implies $H(\nu)$ is finite. Finally the uniqueness of $\nu$ follows from the linearity of $F$, convexity of $U$, and that $\cW_{c_{h},\varepsilon}$ is strictly convex by Lemma \ref{lemma convexity of the transport problem}.
\end{proof}

\begin{rem}
Note that the strict convexity of the regularisation functional allowed us to easily ensure uniqueness of the minimiser in Proposition \ref{lemma : well-posedness of jko scheme}.
\end{rem}


\section{Proof of the Main Result}
\label{sec:GeneralCase}
This section presents the proof of the main result, Theorem  \ref{theorem MAIN THEOREM}. We first establish discrete Euler-Lagrange equations for the minimisers of the regularised scheme \ref{eq: regularized scheme}, then we derive necessary a priori estimates, and finally we prove the convergence \textcolor{black}{(up to a subsequence)} of the scheme.

\subsection{Discrete Euler-Lagrange Equations}\label{section E.L eq}

In this section we study the minimisers of the optimisation problem \eqref{eq : minimisation problem}. This is done by studying the functional $\frac{1}{2h}\cW_{c_h,\varepsilon}\big(\mu,\cdot \big)+\cF(\cdot)$ (for a fixed $\mu \in \cP^r_2(\bR^d)$) at small perturbations around its minimiser. Recall that Proposition \ref{lemma : well-posedness of jko scheme} ensured well-posedness of \eqref{eq : minimisation problem} for small enough $h,\varepsilon>0$, and thus the associated Euler-Lagrange equations will also hold for such $h,\varepsilon$ small enough.

When \eqref{eq:Formal-Evolution} is describing a Wasserstein gradient flow its solution can be viewed as the minimiser of a large deviation rate functional \cite{adams2013large}. With this perspective one can view the Euler-Lagrange equations, established below in Lemma \ref{lemma : general E.L equation}, as the discrete analogue of \eqref{eq: weak formulation general PDE}.  

Throughout this section, for a given vector field $\eta\in C^\infty_c(\bR^d;\bR^d)$  we call $\Phi : \bR_+\times \bR^d \to \bR^d$ the flow through $\eta$ with dynamics
\begin{equation}\label{eq : the flow}
    \partial_s \Phi_s = \eta(\Phi_s),~\Phi_0=\text{id}.
\end{equation}

The following result is well established (for instance see \cite[Proposition 3.5]{carlier2017convergence}). 
\begin{lemma}
Let $\nu\in \cP_2^r(\bR^d)$, and $\eta\in C^\infty_c(\bR^d;\bR^d)$ with flow $\Phi_s$ defined in \eqref{eq : the flow}. The first variation of the free energy $\cF$ (associated with \eqref{eq : minimisation problem}) at $\nu$ along $\eta$, and denoted by $\delta \cF(\nu,\eta)$, is 
\begin{equation}
\label{eq:aux:gradcalF}
    \delta \cF(\nu,\eta)
    :=\frac{d}{d s} \cF\big((\Phi_{s})_{\#}\nu\big)\Big|_{s=0}
    =\int_{\bR^d} \nu(y) \Big\langle \eta(y),  \nabla f(y) \Big\rangle dy-\int_{\bR^d} p(\nu(y)) \divv (\eta (y)) dy.
\end{equation}
\end{lemma}

\begin{lemma}[Euler-Lagrange equation]\label{lemma : general E.L equation} Let $\mu\in\cP_2^r(\bR^d)$, and $h,\varepsilon$ be small enough.
 Let $\nu$ be the optimum in \eqref{eq : minimisation problem}, and let $\gamma$ be the corresponding  optimal plan in $\cW_{c_h,\varepsilon}(\mu,\nu)$. 
 Then, for any $\eta \in C_c^\infty(\bR^d;\bR^d)$ we have
\begin{equation}
  0=\frac{1}{2h}\int_{\bR^{2d}} \Big\langle \eta(y), \nabla_y   c_h(x,y)  \Big\rangle  d\gamma(x,y) - \frac{\varepsilon}{2h}\int_{\bR^d}  \nu(y) \divv (\eta (y)) dy  +\delta \cF(\nu,\eta).
  \label{eq Euler-Lagrange}
\end{equation}
In particular, by \eqref{equation assumption for the cost}, we have for any function $\varphi \in C^\infty_c(\bR^d)$ 
\begin{align}
\nonumber
  \frac{1}{h}\int_{\bR^{2d}} \Big\langle (y&-x) , \nabla \varphi(y) \Big\rangle 
  d\gamma(x,y)
  \\ \nonumber
  &
  =\int_{\bR^d} \nu(y)\Big\langle b(y) , \nabla \varphi(y) \Big\rangle dy +\frac{\varepsilon}{2h}\int_{\bR^d}  \nu(y) \divv \big( (A+B_h)\nabla \varphi (y) \big) dy 
  \\
  &\qquad 
  -\delta \cF(\nu,(A+B_h)\nabla \varphi)+ O(h)(1+\|\nabla \varphi \|_{\infty})\Big(M(\mu)+M(\nu)+1\Big)   +  O\big(\frac{1}{h}\big)(c_h,\gamma)
  \label{eq Euler-Lagrange nice form}
\end{align}
\end{lemma}

\begin{proof}
Let $\Phi$ be defined as in \eqref{eq : the flow}. Since $\nu$ is optimal for the minimisation problem \eqref{eq : minimisation problem} we have
\begin{align}
      \frac{1}{2h}\cW_{c_h,\varepsilon}(\mu,\nu)+\cF(\nu) \leq&  \frac{1}{2h}\cW_{c_h,\varepsilon}(\mu,(\Phi_s)_{\#}\nu)+\cF((\Phi_s)_{\#}\nu),
      \nonumber
\end{align}
which implies, 
      \begin{align} 
      0\leq&  \limsup_{s\to 0}  \frac{1}{2hs}\Big(\cW_{c_h,\varepsilon}(\mu,(\Phi_s)_{\#}\nu)-\cW_{c_h,\varepsilon}(\mu,\nu)\Big) + \delta \cF(\nu,\eta).
      \label{eq 78}
\end{align}
Let $\gamma$ be the optimal coupling in \eqref{eq : minimisation problem}. Then, for $\tilde{\Phi}_s:=(\text{id},\Phi_s)$, we know $(\tilde{\Phi}_s)_{\#}\gamma \in \Pi^r(\mu,(\Phi_s)_{\#}\nu)$  so we have 
\begin{align*}
    \limsup_{s\to 0}  \frac{1}{2hs}\Big(\cW_{c_h,\varepsilon}(\mu,(\Phi_s)_{\#}\nu)
    &
    -\cW_{c_h,\varepsilon}(\mu,\nu)\Big) 
    \\
    &
    \leq \limsup_{s\to 0}  \frac{1}{2hs}\Big( (c_h,(\tilde{\Phi}_s)_{\#}\gamma) -(c_h,\gamma) + \varepsilon \Big( H((\tilde{\Phi}_s)_{\#}\gamma)-H(\gamma)\Big) \Big).
    \end{align*}
    By Fatou's Lemma we have 
    \begin{align*}
         \limsup_{s\to 0}  \frac{(c_h,(\tilde{\Phi}_s)_{\#}\gamma) -(c_h,\gamma) }{2hs} \textcolor{black}{\leq} \frac{1}{2h} \int_{\bR^{2d}} \Big\langle \eta(y) , \nabla_y c_h(x,y)  \Big\rangle  d\gamma(x,y),
    \end{align*}
    and also
    \begin{align*}
   \limsup_{s\to 0}  \frac{\varepsilon \Big( H((\tilde{\Phi}_s)_{\#}\gamma)-H(\gamma)\Big) }{2hs}
    \leq&  
    \limsup_{s\to 0}  \frac{-\varepsilon}{2hs}\Big( \int_{\bR^{d}}  \log(|\dett D \Phi_s(y) |) -\log(|\dett D \Phi_0(y) |)d\nu(y) \Big)
    \\
     =&
     -\frac{\varepsilon}{2h}  \int_{\bR^{2d}} \nu(y)\divv \big(\eta(y) \big) dy.
\end{align*}
Injecting this result into \eqref{eq 78} and substituting $\eta$ for $-\eta$ gives the result. 
\end{proof}

\subsection{A priori estimates}\label{A-priori estimates}

In this section we provide a number of a priori estimates which will help to establish the compactness arguments of Section \ref{section the limiting procedure}. Throughout this section the results hold for each fixed $k\in \bN$, that is, for each $h_k,\varepsilon_k,N_k$ of the sequences satisfying \eqref{eq: assumption on scaling}, and the sequence $\{\rho^n_{h_k,\varepsilon_k}\}_{n=0}^{N_k-1}$ built from the scheme \eqref{eq: regularized scheme} with the associated sequence of optimal couplings $\{\gamma^n_{h_k,\varepsilon_k}\}_{n=1}^{N_k}$. For notational convenience we omit the dependence on $k$ and simply write $h,\varepsilon,N,\{\rho^n\}_{n=0}^{N-1},\{\gamma^n\}_{n=1}^{N}$.

\begin{lemma}\label{lemma : change of variable satisfy assumption}
For all $n\in \{1,\ldots, N\}$, we have
\begin{align}
    (c_h,\gamma^n) \leq& Ch^2 \Big( M(\rho^{n-1})+1 \Big) \label{kramer : eq : equation 59, with external}
    -
     \varepsilon H(\rho^n) + 2h\Big(\cF(\rho^{n-1})-\cF(\rho^n)\Big), 
\end{align}
for $C>0$ a constant depending only on $\rho_0$ and the constants in the assumptions.
\end{lemma}

In the well established JKO procedure \cite[Eqs.~(42)-(45)]{jordan1998variational} one compares $\frac{1}{2h}W^2_{2}(\rho^{n-1},\rho^n)+\cF(\rho^n)$ against $\frac{1}{2h}W^2_{2}(\rho^{n-1},\rho^{n-1})+\cF(\rho^{n-1})$. The term $W^2_{2}(\rho^{n-1},\rho^{n-1})$ is zero, and hence one would end up with a control of \textcolor{black}{$W_2(\rho^{n-1},\rho^n)$} in terms of the free energy. However, \textcolor{black}{in the present work}, since $\cW_{c_{h},\varepsilon}$ is not a metric, we need to pick a new distribution to compare the performance of $\rho^n$ against. We judiciously choose such a distribution as to make the cost $c_h$ of transporting mass free. 
\begin{proof}
This proof has two steps. First, is the choice of the distribution $\rho_\sigma$ against which to compare $\rho^n$. 
The second part is carrying out the said comparison.

\textit{Step 1: the candidate distribution $\rho_\sigma$ and its properties.} 
Let $G\in C^\infty_c (\mathbb{R}^d)$ be a probability density, such that $M(G)=1$, $H(G) <\infty$. For a scaling parameter $\sigma>0$, to be chosen later, define $G_\sigma(\cdot):=\sigma^{-d}G(\frac{\cdot}{\sigma})$. For $\cT_h$ defined in Assumption \ref{Assumption change of var for f and c - for regular.} define
\begin{equation*}
   \gamma_\sigma(x,y) 
   := 
   \rho^{n-1}(x) G_\sigma\big(y-\mathcal{T}_h(x)\big),
\end{equation*}
as a joint distribution with first marginal $\rho^{n-1}$, and second marginal 
\begin{equation*}
\rho_\sigma(y):=\int \gamma_\sigma(x,y) dx.
\end{equation*}
Then, the change of variables $y=\mathcal{T}_h(x)+\sigma z$ and leaving $x$ unchanged, has Jacobian 
\begin{equation}
  J(x,z):=  \begin{pmatrix} D \mathcal{T}_h(x) & \sigma\\
1 & 0
\end{pmatrix},
\end{equation}
with determinant $\big|\dett J(x,z)\big|=\sigma^{d}$. Where the entries $\sigma,1,0$ are  $d\times d$-dimensional matrices of that entry multiplied by the identity matrix. Applying the change of variable and calculating we have 
\begin{align}
\nonumber 
    (c_h,\gamma_\sigma)=& \int_{\bR^d} c_h(x,y)  \rho^{n-1}(x) G_\sigma(y-\mathcal{T}_h(x)) dxdy
    \\
    =& \int_{\bR^d} c_h(x,\mathcal{T}_h(x)+\sigma z)  \rho^{n-1}(x)  G(z)  dxdz.
\end{align}
Hence by Assumption \ref{Assumption change of var for f and c - for regular.}, it follows 
\begin{align}
\nonumber
    (c_h,\gamma_\sigma)
    \leq& C \int_{\bR^{2d}}\Big( \frac{\sigma}{h^\beta }\Big(\|z\|^2+1\Big) +h^2\Big(\|x\|^2+1\Big)\Big)\rho^{n-1}(x) G(z) dxdz
    \\
    \nonumber
    =& C \Big( \frac{\sigma}{h^\beta}\Big(\int_{\bR^d}\|z\|^2G(z)dz+1\Big)+h^2 \Big(\int_{\bR^{2d}}\|x\|^2 \rho^{n-1}(x)dx+1 \Big) \Big)
    \\
    =&C\Big(\frac{\sigma}{h^\beta }+h^2 \big(M(\rho^{n-1})+1\big) \Big) \label{kramer : eq : equation 60, with external}.
\end{align}
Moreover, a straightforward calculation gives
\begin{align}
\label{kramer : eq : equation 65, with external}
    H(\gamma_\sigma)=&H(\rho^{n-1})-d\log \sigma+H(G).
\end{align}
Again by Assumption \ref{Assumption change of var for f and c - for regular.} and the change of variables above we have the following estimate for the potential energy 
\begin{align}
    F(\rho_\sigma)=&\int_{\bR^d} f(y) \rho_\sigma(y) dy\notag
    \\
    \leq& \int_{\bR^{2d}}\Big(  |f(y)-f(x)|+f(x)\Big) \rho^{n-1}(x) G_\sigma(y-\mathcal{T}_h(x)) dxdy\notag
    \\
    =&  \int_{\bR^{2d}} \Big(  |f(\cT_h(x)+\sigma z)-f(x)|\Big) \rho^{n-1}(x) G(z) dxdy + \int_{\bR^{2d}} f(x) \rho^{n-1}(x) G(z) dxdz\notag
    \\
    \leq& C \int_{\bR^{2d}} \Big( \frac{\sigma}{h^\beta }\Big(\|z\|^2+1\Big) +h\Big(\|x\|^2+1\Big)\Big)\rho^{n-1}(x) G(z) dxdz+  F(\rho^{n-1})\notag
    \\
    \leq&  C \Big(   \frac{\sigma}{h^\beta } + h\Big(  M(\rho^{n-1})+1 \Big) \Big)+F(\rho^{n-1}).
    \label{eq 41}
\end{align}
Jensen's inequality implies (by the convexity of $u$) that for the internal energy 
\begin{equation}\label{eq 40}
    U(\rho_\sigma)=\int_{\bR^d}u\Big(\int_{\bR^d} \gamma_{\sigma} (x,y) dx\Big) dy \leq \int_{\bR^{2d}} u(\rho^{n-1})G_{\sigma}(y-\cT_h(x)) dx dy= U(\rho^{n-1}).
\end{equation}
Therefore, plugging \eqref{eq 41} and \eqref{eq 40} together yields 
\begin{align}
\label{kramer : eq : equation 64, with external}
\nonumber 
\cF(\rho_\sigma)
    & \leq  
    C \Big(   \frac{\sigma}{h^\beta } + h\Big(  M(\rho^{n-1})+1 \Big) \Big)+F(\rho^{n-1})+U(\rho^{n-1})
    \\
    &
    = C \Big(   \frac{\sigma}{h^\beta } + h\Big(  M(\rho^{n-1})+1 \Big) \Big) +\cF(\rho^{n-1}).
\end{align}

\textit{Step 2: comparing $\rho_\sigma$ and $\rho^n$.} Since the $\{\rho^n\}$ are built from the scheme \eqref{eq: regularized scheme}, and $\gamma_\sigma$ is a coupling of $\rho^{n-1}$ and $\rho_\sigma$, we have 
\begin{align}
\label{kramer : eq : equation 71, with external}
    \frac{1}{2h}\Big((c_h,\gamma^n)+\varepsilon H(\gamma^n) \Big)  +\cF(\rho^n) 
    \leq  
    \frac{1}{2h} \cW_{c_h,\varepsilon}(\rho^{n-1},\rho_\sigma)+\cF(\rho_\sigma)
    \leq  
    \frac{1}{2h}\Big((c_h,\gamma_\sigma) + \varepsilon H(\gamma_\sigma) \Big) + \cF(\rho_\sigma).
\end{align}
Substituting the above calculations   \eqref{kramer : eq : equation 60, with external}, \eqref{kramer : eq : equation 65, with external} and \eqref{kramer : eq : equation 64, with external} into \eqref{kramer : eq : equation 71, with external} we get 
\begin{align}
\nonumber 
       \frac{1}{2h}\Big((c_h,\gamma^n)+\varepsilon H(\gamma^n) \Big)  +\cF(\rho^n)
       \leq
       & \frac{1}{2h} \Big( C\Big(\frac{\sigma}{h^\beta }+h^2 \Big(M(\rho^{n-1})+1\Big) \Big) + \varepsilon \Big( H(\rho^{n-1})-d\log \sigma+H(G) \Big) \Big)
       \\
       &+ C \Big(   \frac{\sigma}{h^\beta } + h\Big(  M(\rho^{n-1})+1 \Big) \Big)+\cF(\rho^{n-1}).
\end{align}
Rearranging the terms and using that $H(\gamma^n)\geq H(\rho^n)+H(\rho^{n-1})$ we obtain 
\begin{align}
\nonumber 
       (c_h,\gamma^n)  \leq&  C\Big(\frac{\sigma}{h^\beta }+h^2 \Big(M(\rho^{n-1})+1\Big) \Big) + \varepsilon \Big( - H(\rho^{n})-d\log \sigma+H(G) \Big)
       \\
       &+2h C \Big(   \frac{\sigma}{h^\beta } + h\Big(  M(\rho^{n-1})+1 \Big) \Big)+2h\Big( \cF(\rho^{n-1}) - \cF(\rho^n) \Big).
\end{align}

Now we are free to choose \textcolor{black}{$\sigma = \varepsilon^{1+\frac{\beta}{2}}$}. Recall that the scaling \eqref{eq: assumption on scaling} implies $ \frac{\sigma}{h^\beta }\leq Ch^2$ and $-\varepsilon d\log\sigma \leq (1+\frac{\beta}{2})  \varepsilon d\log |\varepsilon| $, we thus have
\begin{align*}
       (c_h,\gamma^n)  \leq&  Ch^2 \Big(M(\rho^{n-1})+1 \Big) 
    - \varepsilon H(\rho^{n})+2h\Big( \cF(\rho^{n-1}) - \cF(\rho^n) \Big).
\end{align*}
\end{proof}

From Lemma \ref{lemma : change of variable satisfy assumption} we are able to establish uniform boundedness of the 2nd moment, energy and entropy, of the solutions to the variational scheme \eqref{eq: regularized scheme}. This is the result we present next. One should note that in the following bounds the constant $C$ depends on  the dimension $d$, the constants of our assumptions, the initial data $\rho^0$, but importantly is independent of $k$. We mention that the following proof differs from classical a-priori bounds for a JKO scheme since $c_h$ is not assumed to be a metric. We follow a similar strategy to that found in \cite{duong2014conservative,Huang00}, first obtaining bounds locally and then extending them over the full time interval.

\begin{lemma}[Bounded Moments, Energy, and Entropy]
\label{kramer : lemma : a-priori bounds general}
For small enough $h,\varepsilon>0$, we have for all $n \in \{ 1,\ldots,N\}$ 
\begin{align}
\label{eq momments energy bound}
    M(\rho^n),\cF(\rho^n),- H(\rho^n)<C.
\end{align}
\begin{proof}

\color{black}
We begin by finding an $h_0,T_0$ independent of the initial data, and a $C_0$ depending only on $M(\rho^0),\cF(\rho^0)$ such that
\begin{align}
    M(\rho^n),\cF(\rho^n),- H(\rho^n)<C_0.
\end{align}
holds for all $n\leq \big\lceil \frac{T_0}{h} \big\rceil$ with $h\leq h_0$. Now for any $i\in \{1,\ldots,N\}$

\begin{align}
    M(\rho^i)^{\frac{1}{2}} \leq&  M(\rho^{i-1})^{\frac{1}{2}}+ W_2(\rho^{i-1},\rho^{i})
    \label{eq change1}
    \\
   \leq&  M(\rho^{i-1})^{\frac{1}{2}}+C \Big( (c_h,\gamma^i)+h^2(M(\rho^{i-1})+M(\rho^i) \Big)^{\frac{1}{2}}
   \label{eq change2}
    \\
   \leq&  M(\rho^{i-1})^{\frac{1}{2}}+C \Big( (c_h,\gamma^i)^{\frac{1}{2}}+h(M(\rho^{i-1})^{\frac{1}{2}}+M(\rho^i)^{\frac{1}{2}}) \Big),
   \nonumber
\end{align}
where in \eqref{eq change1} we have used the Minkowski integral inequality, and in \eqref{eq change2} we have used Lemma \ref{lemma :  our cost and the wasserstein}. Summing over $i=1,\ldots, n$, and denoting $M^0=M(\rho^0)$ we get 
\begin{align}
    M(\rho^n)^{\frac{1}{2}} \leq&  C    \Big( (M^0)^{\frac{1}{2}} + \sum_{i=1}^n (c_h,\gamma^i)^{\frac{1}{2}}+h\sum_{i=1}^n M(\rho^{i})^{\frac{1}{2}} \Big).
    \label{eq change4}
\end{align}
Squaring \eqref{eq change4}, and then using Cauchy–Schwarz inequality we get 
\begin{align*}
    M(\rho^n)\leq&  C    \Big( M^0 + \big(\sum_{i=1}^n (c_h,\gamma^i)^{\frac{1}{2}}\big)^2+h^2\big(\sum_{i=1}^nM(\rho^{i})^{\frac{1}{2}}\big)^2\Big)
    \\
    \leq& C  \Big( M^0 + n\sum_{i=1}^n (c_h,\gamma^i)+h^2 n \sum_{i=1}^nM(\rho^{i})\Big).
\end{align*}
Now applying Lemma \eqref{lemma : change of variable satisfy assumption}, and recalling $N h =T$,  we have 

\begin{align*}
    M(\rho^n)\leq&  
     C  \Big( M^0 - n \varepsilon\sum_{i=1}^n
      H(\rho^i) + 2hn\Big(\cF(\rho^0)-\cF(\rho^n)\Big)
    +h  \sum_{i=1}^nM(\rho^{i})\Big),
\end{align*}
Next recalling that $f$ is positive, and using Lemma \ref{lemma : appendix : entropy bound} twice, we can deduce 
\begin{align}
\label{eq change5}
    M(\rho^n)\leq&  
     C_1  \Big( C_0 + \varepsilon n\sum_{i=1}^n  (1+M(\rho^i))^{\alpha} + (1+M(\rho^n))^{\alpha}
    + h  \sum_{i=1}^nM(\rho^{i})\Big),
\end{align}
for some fixed constant $C_0>0$ depending only on $M(\rho^0),\cF(\rho^0)$, and a fixed the constant $C_1>0$ independent of the initial condition.
Fixing a time horizon $T_0$ small enough, and $N_0:=\big\lceil \frac{T_0}{h} \big\rceil $, we let $h_0$ be such that for all $h \leq h_0$, $N_0 h \leq  2T_0$. Therefore, for all $h\leq h_0$, and any $n_0 \leq N_0$, summing \eqref{eq change5} over $n=1,\ldots,n_0$, 
\begin{equation*}
   \sum_{n=1}^{n_0}  M(\rho^n)\leq 
     C_1  \Big( n_0 C_0 +  (n_0^2\varepsilon+1) \sum_{n=1}^{n_0}   (1+M(\rho^{n}))^{\alpha} 
    + h n_0 \sum_{n=1}^{n_0} M(\rho^{n})\Big),
\end{equation*}
Choosing $T_0$ small enough that $C_1 h N_0 \leq \frac{1}{2}$, one can see that
\begin{equation}
\label{eq change6}
  \frac{1}{2} \sum_{n=1}^{n_0}  M(\rho^n)\leq  
     C_1  \Big( n_0 C_0 +  (n_0^2\varepsilon+1) \sum_{n=1}^{n_0}   (1+M(\rho^{n}))^{\alpha}\Big).
\end{equation}
Substituting \eqref{eq change6} into the last term in \eqref{eq change5} we have, for all $n \leq N_0$ and $h \leq h_0$,
\begin{align}
    M(\rho^n)\leq& C_1\Bigg(C_0+ \varepsilon n\sum_{i=1}^n  (1+M(\rho^i))^{\alpha} + (1+M(\rho^n))^{\alpha}+
    2 h    C_1  \Big( n C_0 +  (n^2\varepsilon+1) \sum_{n=1}^{n}   (1+M(\rho^{n}))^{\alpha}\Big) \Bigg).
     \nonumber
\end{align}
Using $nh\leq T$ in conjunction with the scaling \eqref{eq: assumption on scaling}, specifically $\varepsilon \leq Ch^2$, the above inequality simplifies to 
\begin{align}
\label{eq change8}
    M(\rho^n)\leq&  
     \tilde{C}_1  \Big( \tilde{C}_0 +h\sum_{i=1}^n  (1+M(\rho^i))^{\alpha} \Big),
\end{align} 
for some new fixed constant $\tilde{C}_0>0$ depending only on $M(\rho^0),\cF(\rho^0)$, and a fixed the constant $\tilde{C}_1>0$ independent of the initial condition. Let $\bar{M}=\max_{n \leq N_0} M(\rho^n)$. Since \eqref{eq change8} holds for all $n\leq N_0$, this implies that

\begin{align}\label{eq z12}
    \bar{M}\leq& 
     \tilde{C}_1  \Big( \tilde{C}_0 +hN_0 (1+\bar{M})^{\alpha} \Big).
\end{align}
Choose $T_0$ small enough that $\tilde{C}_1 hN_0\leq \frac{1}{2}$. From \eqref{eq z12} we can use the Bernoulli inequality to claim, for some new fixed constant $C_0>0$ depending only on $M(\rho^0),\cF(\rho^0)$, that for all $h \leq h_0,n\in\{0,\ldots,N_0\}$ with $N_0=\lceil \frac{T_0}{h} \rceil$ 
\begin{equation}
\label{eq change11}
    M(\rho^n)\leq C_0,-H(\rho^n)\leq C_0,
\end{equation}
where we recall that $T_0,h_0$ are all independent of the initial condition. 
Now we obtain a similar bound for $\cF(\rho^n)$. Returning to Lemma \ref{lemma : change of variable satisfy assumption}, and using the non-negativity of $c_h$, we see that for any $i\in\{1,\ldots,N\}$
 \begin{align}
 \label{eq change10}
     h\big(\cF(\rho^i)-\cF(\rho^{i-1})\big)  \leq& Ch^2\big(1 +M(\rho^i)\big) -\varepsilon H(\rho^i) .
     \end{align}
     Upon rearranging \eqref{eq change10}, employing \eqref{eq appendix 1}, and using the Bernoulli inequality, we get that 
     
     \begin{align*}  \cF(\rho^i)-\cF(\rho^{i-1})\leq& Ch\big(1+M(\rho^i) \big).
\end{align*}
Summing the above inequality over $i=1,\ldots,n \leq N_0$ yields 
\begin{align*}  \cF(\rho^n)\leq& C h \sum_{i=1}^n \big(1+M(\rho^i) \big) + \cF( \rho^0).
\end{align*}
Now  we can use \eqref{eq change11} ,  and that $h N\leq T$ to obtain
\begin{equation}\label{eq uniform moments new proof}
    \cF(\rho^n)\leq C_0 
\end{equation}
for all $n\leq N_0$. Since the $T_0$ and $h_0$ we have chosen are independent of the initial data we can extend the bound \eqref{eq uniform moments new proof} to all $n\in \{1,\ldots, N\}$ similarly as has been done in \cite[Lemma 5.3]{Huang00}, see also \cite{duong2014conservative}, which completes the proof.

\color{black}
\end{proof}
\end{lemma}

\begin{corollary}[The total sum of the costs]\label{corollary sum of the costs} Let $h$ be sufficiently small, then we have 
\begin{equation*}
    \sum_{i=1}^{N} (c_h,\gamma^n) \leq C h.
\end{equation*}

\begin{proof}
Summing \eqref{kramer : eq : equation 59, with external} over $n$, using the bounds of Lemma \ref{kramer : lemma : a-priori bounds general}, and the scaling Assumption \ref{assumption on scaling} yields the result.
\end{proof}
\end{corollary}

\subsection{The limiting procedure}\label{section the limiting procedure}
Let $\{\rho^n_{h_{k},\varepsilon_{k}}\}_{n=0}^{N_k}$ be the solution of our scheme \eqref{eq: regularized scheme} with associated optimal plans $\{\gamma^n_{h_k,\varepsilon_k}\}_{n=1}^{N_k}$, and interpolation $\rho_k$ defined in \eqref{eq: time interpolation}. For notational convenience throughout this section we write $\rho^n_{h_k,\varepsilon_k}=\rho^n_k$, $\gamma^n_{h_k,\varepsilon_k}=\gamma^n_k $. 
As is common in the JKO procedure, the a priori estimates give us enough compactness to pass, at least along a subsequence, to the limit of $\rho_k$ to some $\rho$ in $L^1((0, T)\times \bR^d)$. We show that $\rho$ is in fact a weak solution of \eqref{eq:Formal-Evolution}. 

\begin{lemma} The sequence of interpolations $\rho_k : [0,T] \times \bR^d \to \bR$ constructed from \eqref{eq: time interpolation} satisfy for any $\varphi \in C^\infty_c(\bR^d)$.
\begin{equation}
\label{eq 92}
 \int_0^T \int_{\bR^d} \rho_k(t,x) \Big( \frac{\varphi(t+h_k,x)-\varphi(t,x)}{h_k} \Big)dxdt
 = 
 -\int_0^{h_k} \int_{\bR^d}\rho^0(x)
  \frac{\varphi(t,x)}{h_k} dx dt  +   Q_k + R_k +O(h_k),
\end{equation}
where 
\begin{align}
     Q_k=&\int_0^T \int_{\bR^{d}} \rho_k(t,y) \Big(  \Big\langle\nabla f(y)  , \Big(A+B_{h_k}\Big)\nabla \varphi (t,y) \Big\rangle - \Big\langle b(y) , \nabla \varphi(t,y) \Big\rangle - \frac{\varepsilon_k}{2h_k}\divv \Big(\Big(A+B_{h_k}\Big)\nabla \varphi (t,y)\Big) \Big) dydt,
     \label{eq 101}
    \\
    R_k=& -\int_0^T \int_{\bR^{d}}p(\rho_k(t,y)) \divv \Big(\Big(A+B_{h_k}\Big)\nabla \varphi (t,y)\Big)  dydt.
    \label{eq 102}
\end{align}
\end{lemma}
\begin{proof}
Again, for notational convenience, we write $h_k=h,\varepsilon_k=\varepsilon,N_k=N$ omitting the dependence on $k$ but leave the dependence explicit in $\gamma_k$ and $\rho_k$. Let $t\in[0,T]$, the Taylor expansion yields 
\begin{align}
\nonumber
\int_{\bR^d}\Big( \rho^{n}_k(x)-\rho^{n-1}_k(x) \Big) \varphi(t,x) dx
& 
=\int_{\bR^{2d}} \Big( \varphi(t,y)-\varphi(t,x) \Big) d\gamma^n_k(x,y) 
\\
&=
\int_{\bR^{2d}} \Big\langle y-x , \nabla \varphi(t,y) \Big\rangle d\gamma^n_k(x,y)+\kappa_n(t), \label{eq 72}
\end{align}
where the remainder $\kappa_n$ is bounded using \eqref{eq : cost realted to euclidean cost} and Lemma \ref{kramer : lemma : a-priori bounds general}, namely, 
\begin{align}
\nonumber
|\kappa_n(t)| \leq \frac{1}{2} \| \nabla^2 \varphi \|_{\infty} \int_{\bR^{2d}} \|x-y\|^2 d\gamma^n_k(x,y) 
&
\leq C \int_{\bR^{2d}} \Big( c_h(x,y)+h^2\Big(\|x\|^2+\|y\|^2\Big) \Big) d\gamma^n_k(x,y) 
\\ \nonumber
&
= C\Big((c_h,\gamma^n_k)+h^2\Big(M(\rho_k^{n-1})+M(\rho_k^n)\Big)\Big)
\\
&\leq C \Big( (c_h,\gamma^n_k)+ h^2 \Big) .
\label{bound for kappa} 
\end{align}
From \eqref{eq 72} and using \eqref{eq Euler-Lagrange nice form}, whose $O(\cdot)$ terms absorb \eqref{bound for kappa}, we have 

\begin{align}
\nonumber
     \int_{\bR^d}\Big( \frac{\rho^{n}_k(x)-\rho^{n-1}_k(x)}{h} \Big)\varphi(t,x) dx=& \int_{\bR^{2d}} \Big\langle  b(y),\nabla \varphi(t,y) \Big\rangle d\gamma^n_k(x,y)
    \\ 
     \nonumber
     & + \int_{\bR^{d}} \Big( p(\rho^{n}_k(y)) +\frac{\varepsilon}{2h}\rho^{n}_k(y) \Big) \divv \Big(\Big(A+B_h\Big)\nabla \varphi (t,y)\Big) dy
    \\
     \nonumber
    -& \int_{\bR^d}  \rho^n_k(y) \Big\langle \nabla f(y) , \Big(A+B_h\Big)\nabla \varphi (t,y)\Big\rangle dy
    \\
    \label{auxiliary eq with Os}
     +&O(h)(1+\|\nabla \varphi\|_\infty) \Big( M(\rho_k^{n-1}) +M(\rho_k^{n})+1 \Big)  + O\Big(\frac{1}{h}\Big) (c_h,\gamma^n_k).
\end{align}
Integrating over the interval $(t_{n-1},t_n)$, and summing over $n$ leads to  
\begin{align}\label{last equation to take limit}
\nonumber
 & 
 \sum_{n=1}^{N} \int_{t_{n-1}}^{t_n}\int_{\bR^d} 
 \Big( \frac{\rho^n_k(x)-\rho^{n-1}_k(x)}{h} \Big) \varphi(t,x) dxdt
 \\
 \nonumber
 & = \int_0^T \int_{\bR^{d}} \rho_k(t,y)   \Big\langle b(y) , \nabla \varphi(t,y) \Big\rangle dydt
     + \int_0^T\int_{\bR^{d}} \Big( p(\rho_k(t,y)) +\frac{\varepsilon}{2h}\rho_k(t,y) \Big) \divv \Big(\Big(A+B_h\Big)\nabla \varphi (t,y)\Big) dydt
    \\
    & \qquad - \int_{0}^T \int_{\bR^d}  \rho_k(t,y) \Big\langle \nabla f(y) , \Big(A+B_h\Big)\nabla \varphi (t,y) \Big\rangle dydt
     +O(h),
     \\ \nonumber 
     &= -Q_k-R_k+O(h),
\end{align}
where $Q_k$ and $R_k$ given are by \eqref{eq 101} and \eqref{eq 102}. To establish the first equality we used the bounded moments result in Lemma \ref{kramer : lemma : a-priori bounds general}, Corollary \ref{corollary sum of the costs} on the sum of the costs to control for the very last term in \eqref{auxiliary eq with Os} after being summed up over $n$, and have used that $Nh=T$. By summation by parts, the LHS is equal 
\begin{align}
\label{eq 1 convergence LHS} \nonumber
      & \sum_{n=1}^{N_k}\int_{t_{n-1}}^{t_n}\int_{\bR^d}\Big(\frac{\rho^n_k(x)-\rho^{n-1}_k(x)}{h}  \Big) \varphi(t,x) dxdt
      \\
      &\qquad \qquad=-\int_0^h \int_{\bR^d} \rho^0(x) \frac{\varphi(t,x)}{h}dx dt+ \int_0^T \int_{\bR^d} \rho_k(t,x)  \Big( \frac{\varphi(t,x)-\varphi(t+h,x)}{h} \Big) dxdt.
\end{align}
Joining \eqref{last equation to take limit} and \eqref{eq 1 convergence LHS}, and re-arranging gives the result \eqref{eq 92}.
\end{proof}

Inline with the classical strategy developed in \cite{jordan1998variational} we are left to take limits in \eqref{eq 92}. The convergence of the additional terms involving $b,\frac{\varepsilon}{h}$ is easy since they are linear in $\rho_k$ and we have the scaling \eqref{eq: assumption on scaling}. The convergence of the non-linear term is  dealt with in the following Section, after which we conclude the proof of Theorem \ref{theorem MAIN THEOREM}.

\subsubsection*{Strong Convergence of the pressure of $\rho_k$}

We emphasise the weak convergence of $\rho_k$ is not enough to deal with convergence of the non-linear term 
\begin{align*}
\int_0^T \int_{\bR^{d}} 
 p(\rho_k(t,y))  \divv \Big(\Big(A+B_h\Big)\nabla \varphi (t,y)\Big) 
     dy dt.    
\end{align*}
Instead, the convergence of $\rho_k \to \rho$ in $L^m([0,T],\bR^d)$ is obtained via the compactness argument \cite[Theorem 2]{rossi2003tightness} similar to that done in \cite{carlier2017convergence, carlier2017splitting}. Then, \eqref{additional assumptions 3} implies $p$ is continuous from $L^m([0,T],\bR^d)$ to $ L^1([0,T],\bR^d)$ and hence $p(\rho_k) \to p(\rho) $ in $L^1([0,T],\bR^d)$. 
\begin{lemma}\label{Lemma : for the strong convergence}
Consider the sequence of interpolations $\rho_k : [0,T] \times \bR^d \to \bR$ constructed from \eqref{eq: time interpolation}, and $m \in \bN$ introduced in Assumption \ref{Assumption on potential and internal energy}. For $k$ large enough 
we have that
\begin{equation}
    \int_0^T \int_{\mathbb{R}^d} \big( (\rho_k(t,y))^m +\|\nabla (\rho_k(t,y))^m\| \big) dydt \leq C,
\end{equation}
where $C>0$ independent of $k$.

\end{lemma}

\begin{proof}
The estimate of Lemma \ref{kramer : lemma : a-priori bounds general} and \eqref{additional assumptions 2} \textcolor{black}{yield directly} 
\begin{equation*}
    \int_0^T\int_{\bR^d} (\rho_k(t,y))^m dydt \leq C.
\end{equation*}
It remains to show 
\begin{equation}
    \int_0^T\int_{\bR^d} \| \nabla (\rho_k(t,y))^m \|dydt \leq C.
\end{equation}
Omit the dependence on $k$ from $\rho^n_k=\rho^n$ and $\gamma^n_k=\gamma^n$ for this proof.
Set $\mu^n:=\frac{\varepsilon}{2h}\rho^n+p(\rho^n)$ and notice that $\mu^n\in L^1(\bR^d)$ by \eqref{additional assumptions 2} and Lemma \ref{kramer : lemma : a-priori bounds general}. From the Euler-Lagrange equation Lemma \ref{lemma : general E.L equation}  
\begin{align}\label{eq 82}
    \int_{\bR^d} \mu^n(y) \divv(\eta(y)) dy = \frac{1}{2h} \int_{\bR^{2d}} \Big\langle \nabla_y c_h(x,y) , \eta(y) \Big\rangle d\gamma^{n}(x,y) + \int_{\bR^d} \Big\langle \rho^n(y)\nabla f(y) , \eta(y) \Big\rangle dy.
\end{align}
Since $\gamma^n \in \Pi^r(\rho^{n-1},\rho^n)$, by the disintegration of measures Theorem \cite[Theorem 2.28]{ambrosio2000functions} there exists a measure valued map $y\to \gamma^n_y$ such that $\gamma^n=\gamma_y^n \times \rho^n$, so that one can write 
\begin{equation*}
    \int_{\bR^{2d}} \Big\langle  \nabla_y c_h(x,y) , \eta(y)\Big\rangle  d\gamma^{n}(x,y)
    = 
    \int_{\bR^d}\Big\langle  \eta(y) , \Big( \rho^n(y) \int_{\bR^d} \nabla_y c_h(x,y) \gamma_y^n(x) dx \Big)\Big\rangle 
 dy.
\end{equation*}
Note that, for each fixed $h>0$, $y\mapsto  \big(\rho^n(y) \int_{\bR^d} \nabla_y c_h(x,y) \gamma_y^n(x) dx\big) \in L^1(\bR^d)$, since by \eqref{eq 93} and Lemma \ref{kramer : lemma : a-priori bounds general},
\begin{align}
 \int_{\bR^d} \Big|  \rho^n(y) \int_{\bR^d} \nabla_y c_h(x,y) \gamma_y^n(x) dx \Big|dy \leq&    \int_{\bR^{2d}} \| \nabla_y c_h(x,y) \| \gamma^n(x,y) dx dy \notag
  \\
  \leq& C(h) \Big(M(\rho^n)+M(\rho^{n-1})+1\Big)<\infty.\notag
\end{align}
Moreover, since $f$ is differentiable and Lipschitz it is clear that $y \mapsto \rho^n(y)\nabla f(y) \in L^1(\bR^d)$. Hence $\mu^n$ has a weak derivative $\nabla \mu^n \in L^1(\bR^d)$.  
Moreover, we prove next that $\mu^n \in \BV(\bR^d)$, concretely, 
\begin{align}
  \Big| \int_{\bR^d} \mu^n(y) \divv(\eta(y)) dy \Big| 
 \leq &
    \Big| \frac{1}{2h} \int_{\bR^{2d}}\Big\langle  \nabla_y c_h(x,y) ,\eta(y) \Big\rangle  d\gamma^{n}(x,y)dxdy \Big| + C\|\eta\|_\infty
    \label{z1}
   \\
  =&
  \Big| \frac{1}{h} \int_{\bR^{2d}}\Big\langle  \Big((y-x)-hb(y) \Big) , (A+B_h)\eta (y)\Big\rangle  d\gamma^{n}(x,y) \Big|
  \label{z2}
  \\
  &+ \Big|O(h)(1+\|\eta\|_\infty)(M(\rho^{n-1})+M(\rho^n)+1)+ O\Big(\frac{1}{h}\Big) (c_h,\gamma^n) \Big| + C\|\eta\|_\infty, 
 \nonumber
\end{align}
where \eqref{z1} follows using that $f$ is differentiable and Lipschitz, and \eqref{z2} follows by  \eqref{equation assumption for the cost}. Notice now that the moments in \eqref{z2} are finite because of Lemma \ref{kramer : lemma : a-priori bounds general} and the $O(h)$ terms are dominated by a constant $C$. Therefore, 
\begin{align}
\eqref{z2} \leq& \Big| \frac{1}{h} \int_{\bR^{2d}}\Big\langle  \Big((y-x)-hb(y) \Big) , (A+B_h)\eta (y)\Big\rangle  d\gamma^{n}(x,y) \Big|
 \label{z3}
  \\
  &+ 
  O\Big(\frac{1}{h}\Big) (c_h,\gamma^n) + C\Big( 1+\|\eta\|_\infty\Big).
  \nonumber
\end{align}
Consider the first term in  \eqref{z3}
\begin{align}
&\Big| \frac{1}{h} \int_{\bR^{2d}}\Big\langle  \Big((y-x)-hb(y) \Big) , (A+B_h)\eta (y)\Big\rangle  d\gamma^{n}(x,y) \Big|
\nonumber
\\
& \leq
O(1)\|\eta\|_\infty\Big( \frac{1}{h} \int_{\bR^{2d}} \|x-y\| d\gamma^n(x,y)  + \int_{\bR^d}  \| b(y) \| \rho^n(y)dy \Big) 
\label{z6}
 \\
 & 
 \leq  O(1)\|\eta\|_\infty \Big( \frac{1}{h}  \Big( \int_{\bR^{2d}} \|x-y\|^2 d\gamma^n(x,y)\Big)^{1/2}  +  1+\int_{\bR^d} \|y\|^2 \rho^n(y)dy \Big)
 \label{z4}
 \\
 & \leq  O(1)\|\eta\|_\infty  \frac{1}{h} \Big( (c_h,\gamma^n)+ O(h^2)\Big)^{1/2}  + C\|\eta\|_\infty,
 \label{z5}
\end{align}
where: \eqref{z6} is because of Cauchy Schwartz inequality and  that $\|(A+B_h)\eta\|_{\infty}\leq  O(1)\|\eta\|_{\infty}$ when $h<1$. \eqref{z4} follows by Jensen's inequality and Assumption \ref{assumption on b and A}. \eqref{z5} follows by  \eqref{eq : cost realted to euclidean cost} and  Lemma \ref{kramer : lemma : a-priori bounds general}, the constant $C$ depends only on the moment bound and the vector field $b$. We thus have, using the bound \eqref{z5} in conjunction with \eqref{z3}, 
\begin{align}
     \Big| \int_{\bR^d} \mu^n(y) \divv(\eta(y)) dy \Big| \leq&  \|\eta\|_{\infty} O\Big(\frac{1}{h}\Big) \Big( (c_h,\gamma^n)+ O(h^2)\Big)^{1/2}  
     \\
     &+O\Big(\frac{1}{h}\Big) (c_h,\gamma^n) + C\Big( 1+\|\eta\|_\infty\Big).
\end{align}
Since $\mu^n$ has weak derivative $\nabla \mu^n \in L^1(\bR^d)$ we have that
\begin{align}
   \| \nabla \mu^n \|_{L^1(\bR^d)} 
   =& 
   \sup_{\{ \eta \in C^\infty_c(\bR^d;\bR^d)~:~ \sup\|\eta\|\leq 1 \} }   \int_{\bR^d} \mu^n(y) \divv(\eta(y)) dy   
   \\
   \leq& 
   C\Big( \frac{1}{h} \Big( (c_h,\gamma^n)+ O(h^2)\Big)^{1/2} + \frac{1}{h} (c_h,\gamma^n)  +1\Big), 
\end{align}
for some $C>0$. Therefore, by Cauchy Schwartz inequality,  Corollary \ref{corollary sum of the costs}, and the scaling Assumption \ref{assumption on scaling}, we have

\begin{align}
    h\sum_{n=1}^N \| \nabla \mu^n \|_{L^1(\bR^d)} \leq& C\sum_{i=1}^N \Big( (c_h,\gamma^n)+ O(h^2)\Big)^{1/2}+\sum_{n=1}^N (c_h,\gamma^n) +
    TC
    \nonumber
    \\
  \leq& C\sqrt{N}\Big( \sum_{i=1}^N (c_h,\gamma^n)+O(h^2) \Big)^{1/2}+C
     \leq  C\sqrt{Nh} +C \leq C,
     \label{z7}
\end{align}
for a constant $C$ independent of $k$. To finish the proof we provide a sketch of the argument and refer the reader to \cite[Proposition 3.13]{carlier2017convergence} for the full details. One can show that $ \| (\rho^n)^{m-1} \nabla \rho^n \| \leq C\|\nabla \mu^n\|$, so that $(\rho^n)^m\in W^{1,1}(\bR^d)$, with
$$
\|\nabla(\rho^n)^m\|\leq C\|\nabla 
\mu^n \|.
$$
Therefore, using \eqref{z7}
\begin{equation}
    \int_0^T \int_{\bR^d} \|\nabla (\rho_k)^m \| dxdt \leq h\sum_{n=1}^N \int_{\bR^d} \|\nabla (\rho^n)^m \|  dx  \leq C h\sum_{n=1}^N \int_{\bR^d} \| \nabla (\mu^n)^m \| dx \leq C.
\end{equation}
\end{proof}

By Lemma \ref{Lemma : for the strong convergence} we can use the compactness results in \cite[Theorem 2]{rossi2003tightness}. That is, following identically \cite[Proposition 3.14, Lemma 3.15]{carlier2017convergence} we have the following strong convergence (we omit the proof). 
\begin{lemma}\label{lemma : strong convergence results} 
As $k\to \infty$, \textcolor{black}{up to a suitable subsequence if necessary}, we have $\rho_k \to \rho $ in $L^m([0,T],\bR^d)$ and $p(\rho_k) \to p(\rho) $  in $L^1([0,T],\bR^d)$. 
\end{lemma}

\subsection{Proof of the main result}
\label{section proof of the main result}
We are finally in a position to prove the main result.
\begin{proof}[\textbf{Proof of Theorem \ref{theorem MAIN THEOREM}}]
Taking the limit, up to a subsequence if necessary, $k\to \infty$ ($h,\varepsilon \to 0$, $N\to\infty$) in \eqref{eq 92} and using the convergence of Lemma \ref{lemma : strong convergence results} we can argue the convergence of $Q_k$ and $R_k$ in \eqref{eq 92} as follows. 
For $Q_k$ of \eqref{eq 101} we have 
\begin{equation*}
   \lim_{k\to\infty} Q_k = \int_0^T \int_{\bR^{d}} \rho(t,y) \Big(  \Big\langle\nabla f(y)  , A\nabla \varphi (t,y) \Big\rangle - \Big\langle b(y) , \nabla \varphi(t,y) \Big\rangle \Big) dydt,
\end{equation*}
since $b$ is continuous (Assumption \ref{assumption on b and A}), and $\|\nabla f\|$ is uniformly bounded, and we have used the scaling \eqref{eq: assumption on scaling}, namely, ${\varepsilon_k}/{h_k}\to 0$. 

For $R_k$ of \eqref{eq 102} it is clear that
\begin{equation*}
    \lim_{k\to\infty}  R_k = -\int_0^T \int_{\bR^{d}}p(\rho(t,y)) \divv \Big(A\nabla \varphi (t,y)\Big)  dydt 
\end{equation*}
We see that the limit $\rho$ satisfies \eqref{eq: weak formulation general PDE}. 
\end{proof}


\appendix
\section{Appendix}
The following is a well  established result that bounds the entropy of a distribution by its second moment. 
\begin{lemma}\cite[Proposition 4.1]{jordan1998variational}\label{lemma : appendix : entropy bound} There exists a $C>0$ and $0<\alpha<1$ such that 
\begin{equation}\label{eq appendix 1}
    H(\mu)\geq-C(M(\mu)+1)^{\alpha},~~\forall \mu\in \cP_2^r(\bR^d).
\end{equation}
And if $U$ is defined as in Assumption \ref{Assumption on potential and internal energy} then 
\begin{equation}\label{eq appendix 2}
    U(\mu) \geq - C(M(\mu)+1)^\alpha,~~\forall \mu\in \cP_2^r(\bR^d).
\end{equation}
Note $C$ is chosen large enough so that \eqref{eq appendix 1} and \eqref{eq appendix 2} hold simultaneously. 
\end{lemma}

The next result provides lower semi-continuity  for the internal energy and the entropy  functional under bounded moments. 
\begin{lemma}\label{lemma : appendix : l.s.c of internal energy}\cite[Proposition 4.1]{jordan1998variational}
Let $u$ satisfy the Assumption \ref{Assumption on potential and internal energy}, and $U$ be defined as 
\begin{equation}
    U(\mu)=\begin{cases}
        \int_{\bR^d} u(\mu(x))dx~&\text{if}~\mu\in\cP^r(\bR^d)
        \\
       \infty~&\text{otherwise}
    \end{cases}.
\end{equation}
Then $U$ is weakly lower semi-continuous under bounded moments, i.e if $\{\mu_k\}_{k\in \bN}\subset \cP_2(\bR^d)$, $\mu \in \cP_2(\bR^d)$ with $\mu_k\rightharpoonup\mu$, and there exists $C>0$ such that $M(\mu_k),M(\mu)< C$ for all $k\in \bN$, then
\begin{equation}
    U(\mu)\leq \liminf_{k\to\infty}U(\mu_k).
\end{equation}  
\end{lemma}


\section{Verification for the examples}

\subsection{Non-linear diffusion equations}\label{section appendix non-linear diffusion}
\begin{proof}[Proof of proposition \ref{prop : the non-linear diffusion i.e continuity equation}]

By Theorem \ref{theorem MAIN THEOREM} one only needs to check that Assumptions \ref{Assumption on potential and internal energy}, \ref{assumption on b and A},  \ref{assumption for the cost} and \ref{assumption on scaling} hold. The Assumptions  \ref{Assumption on potential and internal energy}, \ref{assumption on b and A} and \ref{assumption on scaling} follow directly from the statement of the proposition and hence their verification is omitted. 

We now check Assumption \ref{assumption for the cost} on the cost function. Clearly  \eqref{eq 93} 
and
\eqref{eq : cost upper bound} 
and  \eqref{eq : cost lower bound} 
hold. Let us now verify $\eqref{eq : cost realted to euclidean cost}$. Let $\lambda_1,\lambda_2,\ldots$ with $0<\lambda_1=h \leq \lambda_2 \leq \ldots $ be the eigenvalues of $A+hI$. Note for all $i=2,\ldots,d$, $\lambda_i=C_i+h$ for some $C_i\geq 0$. Hence $A+hI$ is invertible, with an inverse $(A+hI)^{-1}$ that is symmetric with eigenvalues \textcolor{black}{$\frac{1}{\lambda_1},\frac{1}{\lambda_2},\ldots$}. Since it is symmetric it is diagonalizable and therefore its normalised eigenvectors form an orthonormal basis. Let $v_1,\ldots,v_d$ be the  normalised eigenvectors of $(A+hI)^{-1}$. 
For any $x\in \bR^d$ we can write $x=\sum_{i=1}^d x_i v_i$, where $x_i:=\langle x , v_i \rangle$. 
Now since $\|x\|^2=\sum_{i=1}^d x_i^2$, we have
\begin{align*}
    \big\langle (A+hI)^{-1}x , x \big\rangle =
    & \sum_{i=1}^d \frac{1}{\textcolor{black}{\lambda_i}} x_i^2 
    \geq \frac{1}{\textcolor{black}{\lambda_d}}\|x\|^2\geq \frac{1}{C+2}\|x\|^2,
\end{align*}  
for $h<1$ and some $C>0$, verifying \eqref{eq : cost realted to euclidean cost}. Lastly \eqref{equation assumption for the cost} holds by the symmetry of $A+hI$, \textcolor{black}{where we have taken $B_h=hI$} in \eqref{equation assumption for the cost}. 
To complete the proof it remains only to check the change of variable Assumption \ref{Assumption change of var for f and c - for regular.}. For this take $\cT_h(x)=x$, so that \eqref{kramer : eq : equation 61, with external} holds trivially since $c_h(x,x+\sigma z) \leq \sigma \| (A+hI)^{-1} \| \|z\|^2=\sigma O({h^{-\beta}})\|z\|^2$ for some $\beta>0$. Lastly, \eqref{eq 81} holds with this $\cT_h$ as $f$ is Lipschitz.        
\end{proof}

\subsection{The non-linear kinetic Fokker-Planck (Kramers) equation}\label{section appendix KFPE}

\begin{proof}[Proof of Proposition \ref{proposition KFPE number one}]
By Theorem \ref{theorem MAIN THEOREM} one only needs to check that Assumptions  \ref{Assumption on potential and internal energy}, \ref{assumption on b and A}, \ref{assumption for the cost}, \ref{Assumption change of var for f and c - for regular.}, and \ref{assumption on scaling} hold. The Assumptions  \ref{Assumption on potential and internal energy}, \ref{assumption on b and A}, and \ref{assumption on scaling} follow directly from the statement of the proposition and hence their verification is omitted. 
We now check Assumption \ref{assumption for the cost} on the cost function. Clearly \eqref{eq : cost lower bound} holds. The inequality \eqref{eq : cost upper bound} follows by substituting the estimates \cite[Eq.~(46),(47)]{duong2014conservative}\footnote{
The correct statement of \cite[Eq.~(47)]{duong2014conservative} is 
 $ \|\ddot{\bar{\xi}}\|_2^2\leq C\big(h^{-3}\|q-q'\|^2+h^{-1}\|p-p'\|^2+\|p\|^2+\|p'\|^2\big)$.
} 
into $c_h$, giving $c_h(x,v;x',v')\leq O(h^{-3})(\|x\|^2+\|v\|^2+\|x'\|^2+\|v'\|^2) $. The inequality \eqref{eq 93} is verified by the estimates \cite[Eqs.~(40a),(40b),(41)]{duong2014conservative} in conjunction with \eqref{eq : cost upper bound} just obtained. For \eqref{eq : cost realted to euclidean cost} see \cite[Eqs.~(39b),(39c)]{duong2014conservative}. For \eqref{equation assumption for the cost} we take inspiration from \cite{duong2014conservative}, defining, for any $h>0$,
$$
B_h:=\begin{pmatrix}
    -\frac{h^2}{6} & \frac{h}{2} 
    \\
    -\frac{h}{2} & 0
    \end{pmatrix},
$$
where, in the matrix $B_h$, each entry is a $\tilde{d}\times \tilde{d}$-dimensional matrix of that entry multiplied by the identity matrix.
Then for 
$$
\tilde{\eta}=(A+B_h)\eta,
$$
set $\eta^{1}$ (resp $\eta^{2}$) as the first $\tilde{d} $ components of $\eta$ (resp last $\tilde{d} $ components), and similarly for $\tilde{\eta}$. Then the estimate  \cite[page  2531]{duong2014conservative} 
\begin{align*}
    \Big\langle \nabla_{x'}c_h(x,v;x',v'),\tilde{\eta}^{1} & \Big\rangle +\Big\langle \nabla_{v'}c_h(x,v;x',v'),\tilde{\eta}^{2} \Big\rangle  
    \\
    =&2\Big( \Big\langle x'-x,\eta^{1 } \Big\rangle + \Big\langle v'-v, \eta^{2 } \Big\rangle  - h\Big\langle v', \eta^{1 } \Big\rangle  \Big)
    \\
    &+2\Big\langle h\nabla g(x')+\frac{1}{2} \tau_h(x,v;x',v') , -\frac{h}{2} \eta^{1} + \eta^{2 }  \Big\rangle
     \\
    &+2\Big\langle -h\nabla^2 g(x') v'+\frac{1}{2} \sigma_h(x,v;x',v') ,-\frac{h^2}{6} \eta^{1 } + \frac{h}{2}\eta^{2 }  \Big\rangle,
\end{align*}
where \cite[Eq.~(41)]{duong2014conservative} gives bounds on $\tau_h,\sigma_h$, ensures that \eqref{equation assumption for the cost}  holds.

We now verify Assumption \ref{Assumption change of var for f and c - for regular.} with the change of variables $\cT_h(x,v)=(x+hv,v)$, consider the admissible, in the sense of \eqref{eq KFPE COST 1}, cubic
\begin{equation*}
    \bar{\xi}(t)=x+vt+\Big( \frac{3}{h^2}(x'-x-vh)-\frac{v'-v}{h} \Big)t^2+\Big( \frac{v'+v}{h^2}-\frac{2}{h^3}(x'-x)  \Big)t^3,
\end{equation*}
starting at $(x,v)$ and ending at $(x',v')$. Using Assumption \ref{assumption on hamiltonian force in KFPE} we have
\begin{align*}
    c_h(x,v;x',v') \leq& 2C h \Big( \int_0^h  \| \ddot{\bar{\xi}}(t) \|^2 dt + \int_0^h\|\bar{\xi}(t)\|^2 dt \Big).
   \end{align*}
Note that
\begin{align*}
   h\int_0^h  \| \ddot{\bar{\xi}}(t) \|^2 dt
&\leq h^2\sup_{t\in [0,h} \| \ddot{\bar{\xi}}(t) \|^2
\\
&\leq C\Big(h^2\Big\|\frac{3}{h^2}\Big(x'-x-vh\Big)-\frac{v'-v}{h}\Big\|^2 + h^4\Big\|\frac{v'+v}{h^2}-\frac{2}{h^3}\Big(x'-x\Big)\Big\|^2\Big),
\end{align*}
and
\begin{align*}
    h \int_0^h\|\bar{\xi}(t)\|^2 dt
\leq& h^2 \sup_{t\in [0,h]} \|\bar{\xi}(t)\|^2
    \\
    \leq& Ch^2\Big( \|x\|^2+h^2\|v\|^2+h^4\Big\|\frac{3}{h^2}\Big((x'-x-vh\Big)-\frac{v'-v}{h}\Big\|^2+h^6\Big\|\frac{v'+v}{h^2}-\frac{2}{h^3}(x'-x)\Big\|^2\Big).
\end{align*}
Hence we obtain 
\begin{align*}
    c_h(x,v;x',v') \leq& C\Bigg( h^2\Big\|\frac{3}{h^2}\Big(x'-x-vh\Big)-\frac{v'-v}{h}\Big\|^2 + h^4\Big\|\frac{v'+v}{h^2}-\frac{2}{h^3}\Big(x'-x\Big)\Big\|^2
    \\
    & +h^2\Big( \|x\|^2+h^2\|v\|^2+h^4\Big\|\frac{3}{h^2}\Big(x'-x-vh\Big)-\frac{v'-v}{h}\Big\|^2+h^6\Big\|\frac{v'+v}{h^2}-\frac{2}{h^3}\Big(x'-x\Big)\Big\|^2\Big) \Bigg).
   \end{align*}
So considering $ c_h(x,v;\cT_h(x,v)-(\sigma z, \sigma w))$, we have 
\begin{align*}
    c_h(x,v;\cT(x,v)-(\sigma z, \sigma w))\leq& C\Bigg( h^2\|\frac{3}{h^2}(-\sigma z)-\frac{\sigma w}{h} \|^2+h^4\| \frac{\sigma w}{h^2} - \frac{2}{h^3} \sigma z \|^2 
    \\
    &+ h^2\Big( \|x\|^2+h^2\|v\|^2+h^4\|\frac{3}{h^2}(-\sigma z)-\frac{\sigma w}{h} \|^2+h^6\| \frac{\sigma w}{h^2} - \frac{2}{h^3} \sigma z \|^2 \Big)   \Bigg),
\end{align*}
which proves \eqref{kramer : eq : equation 61, with external}. Lastly the Lipschitz property of $f$ gives \eqref{eq 81}, which completes the verification of Assumption \ref{Assumption change of var for f and c - for regular.}.
\end{proof}

\begin{proof}[Proof of Proposition \ref{proposition KFPE number 2}]
By Theorem \ref{theorem MAIN THEOREM} one only needs to check that Assumptions  \ref{Assumption on potential and internal energy}, \ref{assumption on b and A}, \ref{assumption for the cost}, \ref{Assumption change of var for f and c - for regular.}, and \ref{assumption on scaling} hold. The Assumptions  \ref{Assumption on potential and internal energy}, \ref{assumption on b and A}, and \ref{assumption on scaling} follow directly from the statement of the proposition and hence their verification is omitted.

We now check Assumption \ref{assumption for the cost} on the cost function. The conditions \eqref{eq 93}, 
\eqref{eq : cost upper bound}, 
\eqref{eq : cost lower bound}, 
on $c_h$ are easy to verify. For \eqref{eq : cost realted to euclidean cost} see \cite[Eqs.~(39b),(39c)]{duong2014conservative}. Lastly for \eqref{equation assumption for the cost} we again take inspiration from \cite{duong2014conservative} and define for all $h>0$
$$
B_h:=\begin{pmatrix}
    -\frac{h^2}{6} & \frac{h}{2} 
    \\
    -\frac{h}{2} & 0
    \end{pmatrix},
$$
where again, in the matrix $B_h$, each entry is a $\tilde{d}\times \tilde{d}$-dimensional matrix of that entry multiplied by the identity matrix. One can see from \cite[Eq.~(60)]{duong2014conservative} does ensure that \eqref{equation assumption for the cost} holds. 

For Assumption \ref{Assumption change of var for f and c - for regular.} take $\cT_h(x,v)=(x+hv,v)$, we have
$$
c_h(x,v;\cT_h(x,v)-(\sigma z, \sigma w))=\|h \nabla g(x)-\sigma z\|^2+12\|\frac{1}{2}\sigma w-\frac{1}{h}\sigma z\|^2\leq C\Big( h^2 \|x\|^2+ \|\frac{\sigma}{h} z\|^2+ \|\sigma w\|^2 \Big),
$$
which proves \eqref{kramer : eq : equation 61, with external}. Lastly the Lipschitz property of $f$ gives \eqref{eq 81}, which completes the verification of Assumption \ref{Assumption change of var for f and c - for regular.}.
\end{proof}

\subsection{A degenerate diffusion equation of Kolmogorov-type}\label{section appendix degenerate}

The vector $ \mathbf{b} $ and matrix $\cM$ which define the cost function \eqref{eq mean square derivative cost function} are of the form \begin{equation}
\label{eq: b and M}
\mathbf{b}(h,\rx,\ry)=\begin{pmatrix}y_{1}-x_{1}-\frac{h}{1}x_{2}-...-\frac{h^{n-1}}{(n-1)!}x_{n}\\
\vdots\\
h^{i-1}\Big(y_{i}-\sum_{j=i}^{n}\frac{h^{j-i}}{(j-i)!}x_{j}\Big)\\
\vdots\\
h^{n-1}(y_{n}-x_{n})
\end{pmatrix},
\qquad
\cM=\cM_1 \cM_2^{-1},
\end{equation}
with $\cM_1,\cM_2\in \bR^{\tilde{d}n\times \tilde{d}n}$ given by
\begin{align}
\nonumber
& (\cM_1)_{ki}=\begin{cases}
(-1)^{n-k}\frac{(n+i-1)!}{(k+i-n-1)!}, & \quad \text{if} \quad k+i\ge n+1\\
0 & \quad \text{if} \quad  k+i<n+1,
\end{cases}
\\
\nonumber
&\cM_2=\left[\begin{array}{ccc}
1 & ... & 1\\
\begin{pmatrix}n\\
1
\end{pmatrix} & ... & \begin{pmatrix}2n-1\\
1
\end{pmatrix}\\
\vdots & \vdots & \vdots\\
k!\begin{pmatrix}n\\
k
\end{pmatrix}& ... & k!\begin{pmatrix}2n-1\\
k
\end{pmatrix}\\
\vdots & \vdots & \vdots\\
(n-1)!\begin{pmatrix}n\\
n-1
\end{pmatrix} & ... & (n-1)!\begin{pmatrix}2n-1\\
n-1
\end{pmatrix}
\end{array}\right],
\end{align}
where entry of these matrices is to be understood as a $\tilde{d}-$dimensional matrix that is equal to the entry multiplied but the $\tilde{d}-$dimensional identity matrix. The following matrices will also play an important role in the rest of the section 
\begin{align*}
    &
    J_{1}(h):=\mathrm{diag}(1,h,\cdots,h^{n-1}),
&&D:=\mathrm{diag}(0,\ldots,0,1),
    \\
    &
    J_{2}(h):=\begin{pmatrix}1 & h & \frac{h^{2}}{2!} & \frac{h^{3}}{3!} & \cdots & \frac{h^{n-1}}{(n-1)!}\\
 & h & h^{2} & \frac{h^{3}}{2!} & \cdots & \frac{h^{n-1}}{(n-2)!}\\
 &  & h^{2} & \frac{h^{3}}{1!} & \cdots & \frac{h^{n-1}}{(n-3)!}\\
 &  &  & \ddots & \cdots & \vdots\\
 &  &  &  &  & h^{n-1}
 \end{pmatrix},
&&Q:=\begin{pmatrix}0\\
1 & 0\\
 & 1 & 0\\
 &  & \ddots & \ddots\\
 &  &  & 1 & 0
\end{pmatrix}.
\end{align*}
Omitting the $h$ dependence in $J_1,J_2$ for the sake of clarity, we also define
\begin{align}
\nonumber
&T_1:=(2n-1)J_{1}^{T}\cM J_{1}-2h(J_{1}')^{T}\cM J_{1}-h^{2-2n}J_{1}^{T}\cM J_{2}DJ_{2}^{T}\cM J_{1},
\\
\nonumber
&T_{2}:=(1-2n)J_{2}^{T}\cM J_{1}+h\big((J_{2}')^{T}\cM J_{1}+J_{2}^{T}\,\cM \,J_{1}'\big)-hQJ_{2}^{T}\cM J_{1}
+J_{2}^{T}\cM J_{0}\cM J_{1}, 
\\
\nonumber
&T_3:=(2n-1)J_{2}^{T}\cM J_{2}-2h(J_{2}')^{T}\cM J_{2}+2hQJ_{2}^{T}\cM J_{2}
-h^{2-2n}J_{2}^{T}\cM J_{2}DJ_{2}^{T}\cM J_{2}.
\end{align}
Note that, again, $J_1, J_2, Q,D\in \mathbb{R}^{{\tilde{d} n\times \tilde{d} n}}$. Each entry of these matrices should be understood as a matrix of order $\tilde{d}$ that equals the entry multiplied with the $\tilde{d}$-dimensional identity matrix. 

We now state a series of results from \cite{DuongTran18} which will assist us in proving Proposition \ref{prop : the non-linear diffusion i.e continuity equation}.

\begin{lemma}[Proposition 2 of \cite{DuongTran18}]
\label{prop: aux}
The following assertions hold: (1) $T_1$ is anti-symmetric, (2) $T_2=0$, (3) $T_3$ is anti-symmetric, and (4) $\mathrm{Tr}(DJ_{2}^{T}\cM  J_{2})=n^2\tilde{d}h^{2(n-1)}$.
\end{lemma}
\begin{lemma}[Lemma 4.3 of \cite{DuongTran18}]
\label{lem: InverH2H1}
$J_2^{-1}J_1=J$ where
\begin{equation}
\label{eq: H}
J_{ij}=\begin{cases}
0, \quad \text{if}~ j<i\\
(-1)^{j-i}\frac{h^{j-i}}{(j-i)!},\quad \text{if}~ j\geq i.
\end{cases}
\end{equation}
In particular $J_{ii}=1, \quad J_{ii+1}=-h$ and $J_{ij}=o(h^2)$ for $j\geq i+2$. Note that $J\in\bR^{\tilde{d}n \times \tilde{d}n}$ where $J_{ij}$ should be understood as $J_{ij}I_{\tilde{d}}$.
\end{lemma}

For any $h>0$ define 
\begin{equation}\label{eq : definition of KAPPA}
    \mathcal{K}_h=h^{2n-2}(J_2^T\cM J_1)^{-1}.
\end{equation}

\begin{lemma}[Lemma 4.4 of \cite{DuongTran18}]
\label{lem:K}
For $\mathcal{K}_h$ defined in \eqref{eq : definition of KAPPA} we have 
\begin{equation}
(\mathcal{K}_h)_{ij}=(-1)^{n-j}\frac{h^{2n-i-j}}{(2n-i-j+1)!}.
\end{equation}
In particular, $(\mathcal{K}_h)_{nn}=1$ and $(\mathcal{K}_h)_{ij}=o(h)$ for all $(i,j)\neq (n,n)$. Note also that $\mathcal{K}_h\in \mathbb{R}^{\tilde{d}n\times \tilde{d}n}$ where $(\mathcal{K}_h)_{ij}$ should be understood as $(\mathcal{K}_h)_{ij}I_{\tilde{d}}$.
\end{lemma}
With the use of the preceding lemmas we can prove the convergence of the proposed entropic regularised scheme for the degenerate diffusion of Kolmogorov type, Proposition \ref{proposition Kolmogorov type}.
\begin{proof}[Proof of Proposition \ref{proposition Kolmogorov type}]
By Theorem \ref{theorem MAIN THEOREM} we just need to check Assumptions \ref{Assumption on potential and internal energy}, \ref{assumption on b and A}, \ref{assumption for the cost}, \ref{Assumption change of var for f and c - for regular.}, and \ref{assumption on scaling} hold.

The scaling Assumption \ref{assumption on scaling} and Assumption \ref{Assumption on potential and internal energy}  on the internal and potential energy clearly hold. Similarly, its clear that Assumption \ref{assumption on b and A} on $b,A$ is also satisfied. 

We now show the cost $c_h$  defined in \eqref{eq mean square derivative cost function} satisfies Assumption \ref{assumption for the cost}, with $b,A$ given by  \eqref{eq coef of degenerate diffusion} and $A+B_h =\mathcal{K}_h$ defined in \eqref{eq : definition of KAPPA}. Firstly  for \eqref{eq : cost realted to euclidean cost} we take the result directly from \cite[Lemma 2.3]{DuongTran18}. Moreover, one can see that since $\cM$ is constant and by definition of $c_h$ that \eqref{eq : cost upper bound} holds with $C(h)=h^{2-2n}$. From \cite[Lemma 2.2]{DuongTran17} we know that \eqref{eq : cost lower bound} holds. 

Note we can rewrite $\mathbf{b}$ as 
\begin{align*}
\mathbf{b}(h,\x,\y) &=\begin{pmatrix}y_{1}-x_{1}-\frac{h}{1}x_{2}-...-\frac{^{n-1}}{(n-1)!}x_{n}\\
\vdots\\
h^{i-1}\Big(y_{i}-\sum_{j=i}^{n}\frac{h^{j-i}}{(j-i)!}x_{j}\Big)\\
\vdots\\
h^{n-1}(y_{n}-x_{n})
\end{pmatrix}\\
 & =\begin{pmatrix}y_{1}\\
h y_{2}\\
h^{2}y_{3}\\
\vdots\\
h^{n-1}y_{n}
\end{pmatrix}-\begin{pmatrix}1 & h& \frac{h^{2}}{2!} & \frac{h^{3}}{3!} & \cdots & \frac{h^{n-1}}{(n-1)!}\\
 & h & h^{2} & \frac{h^{3}}{2!} & \cdots & \frac{h^{n-1}}{(n-2)!}\\
 &  & h^{2} & \frac{h^{3}}{1!} & \cdots & \frac{h^{n-1}}{(n-3)!}\\
 &  &  & \ddots & \cdots & \vdots\\
 &  &  &  &  & h^{n-1}
\end{pmatrix}\begin{pmatrix}x_{1}\\
x_{2}\\
x_{3}\\
\vdots\\
x_{n}
\end{pmatrix}
=J_{1}\y-J_{2}\x.
\end{align*}
Therefore, we have
\begin{align*}
c_h(\x,\y) & =h^{2-2n}[\y^{T}J_{1}^{T}-\x^{T}J_{2}^{T}]\,\cM \,[J_{1}\y-J_{2}\x]\nonumber\\
 & =h^{2-2n}\Big[\y^{T}J_{1}^{T}\,\cM\,J_{1}\y-\x^{T}J_{2}^{T}\,\cM\,J_{1}\y-\y^{T}J_{1}^{T}\,\cM\,J_{2}\x+\x^{T}J_{2}^{T}\cM J_{2}\x\Big]\nonumber\\
 & =h^{2-2n}\Big[\y^{T}J_{1}^{T}\,\cM\,J_{1}\y-2\x^{T}J_{2}^{T}\,\cM\,J_{1}\y+\x^{T}J_{2}^{T}\cM J_{2}\x\Big].
\end{align*}
Therefore,
\begin{equation*}
\nabla_{\y} c_h(\x,\y)=2h^{2-2n}J_1^T\cM(J_1\y-J_2\x),
\end{equation*}
so that \eqref{eq 93} holds with $C(h)=h^{2-2n}$. Hence we are left to prove \eqref{equation assumption for the cost}. 
Let $\eta\in \mathbb{R}^{\tilde{d}n}$. We choose $\tilde{\eta}\in\mathbb{R}^{\tilde{d}n}$ such that
\begin{equation*}
\begin{pmatrix}
\tilde{\eta}_1\\
\vdots\\
\tilde{\eta}_n
\end{pmatrix}=\mathcal{K}_h\begin{pmatrix}
 \eta_{1}\\
\vdots\\
\eta_{n}
\end{pmatrix} =\mathcal{K}_h  \eta,
\end{equation*}
where $\mathcal{K}_h$ is given in Lemma \ref{lem:K}, implying that $h^{2-2n}\mathcal{K}_h^T (J_1^TMJ_2)=I$.

Using Lemmas \ref{lem: InverH2H1} and \ref{lem:K}, we compute 
\begin{align*}
\Big \langle \nabla_{\y} c_{h}(\x,\y), \tilde{\eta} \Big \rangle
= \Big\langle \nabla_{\y}c_h(\x,\y), \mathcal{K}_h\eta \Big\rangle 
&
=2\Big[(J_2^{-1}J_1-I)\y\cdot \eta +(\y-\x)\cdot\eta \Big]
\\&=2(\y-\x)\cdot\eta-2h\sum_{i=2}^{n} y_{i}\cdot \eta_{i-1} +O(h^2)\|\y\|. 
\end{align*}
For Assumption \ref{Assumption change of var for f and c - for regular.}, define $\hat{\rx}$ as $\hat{\rx}_i:=\sum_{j=i}^n \frac{t^{j-i}}{(j-i)!}\rx_j$ for $i=1,\dots,n$, and consider the change of variable $\cT_h(\rx)=\hat{\rx}$. Assumption \ref{Assumption change of var for f and c - for regular.} holds with this change of variable and, indeed, one can easily check that 
\begin{align*}
         c_h(\rx,\mathcal{T}_h(\rx)+\sigma \rz)\leq& C h^{2-2n} \sigma^2 \|\rz\|^2,
        \qquad\textrm{and}\qquad
         |f(\cT_h(\rx)+\sigma \rz)-f(\rx)|\leq
         C \| \sigma z_n  \|.
    \end{align*}
\end{proof}
\section*{Acknowledgement} We would like to thank the anonymous referees for their useful suggestions for the improvement of the paper. D.A was supported by The Maxwell Institute Graduate School in Analysis and its Applications, a Centre for Doctoral Training funded by the UK Engineering and Physical Sciences Research Council (grant EP/L016508/01), the Scottish Funding Council, Heriot-Watt University and the University of Edinburgh. M. H. Duong was supported by EPSRC Grants EP/W008041/1 and EP/V038516/1.  G.d.R. acknowledges support from the \emph{Funda{\c c}$\tilde{\text{a}}$o para a Ci$\hat{e}$ncia e a Tecnologia} (Portuguese Foundation for Science and Technology) through the project UIDB/00297/2020 (Centro de Matem\'atica e Aplica\c c$\tilde{\text{o}}$es CMA/FCT/UNL).


\bibliographystyle{abbrv}

\end{document}